                     \def\version{17th August, 2013}                %

\documentclass[reqno,11pt]{amsart} 
\usepackage{srcltx}

\usepackage{amsmath, amsthm, a4, latexsym, amssymb}
\usepackage{color}
\usepackage[colorlinks=true,linkcolor=blue,citecolor=blue]{hyperref}

\def\@rmrk#1#2{\refstepcounter
    {#1}\@ifnextchar[{\@yrmrk{#1}{#2}}{\@xrmrk{#1}{#2}}}

\makeatletter\@addtoreset{equation}{section}\makeatother

 \newfont{\bfit}{cmbxti10 scaled 1200}

\renewcommand{\d}{{\rm d}}
 \newcommand{\e}{{\rm e} }

 \newcommand{\eps}{\varepsilon}
 \newcommand{\supp}{{\rm supp}}

 \newcommand{\eff}{{\text{\rm eff}}}
 \newcommand{\R}{\mathbb{R}}
 \newcommand{\N}{\mathbb{N}}
 \newcommand{\Z}{\mathbb{Z}}

\DeclareMathOperator{\Id}{Id}

 \newcommand{\E}{\mathbb{E}}
 
 \renewcommand{\P}{\mathbb{P}}
 \def\1{{\mathchoice {1\mskip-4mu\mathrm l} 
{1\mskip-4mu\mathrm l}
{1\mskip-4.5mu\mathrm l} {1\mskip-5mu\mathrm l}}}
 
 \newcommand{\Acal}{{\mathcal A}}
 
 \newcommand{\Ccal}{{\mathcal C}}

 \newcommand{\Fcal}{{\mathcal F}}

 \newcommand{\Mcal}{{\mathcal M}}
 \newcommand{\Ncal}{{\mathcal N}}
 \newcommand{\Pcal}{{\mathcal P}}
 \newcommand{\Ocal}{{\mathcal O}}
 
 \newcommand{\smfrac}[2]{{\textstyle{\frac {#1}{#2}}}}

\newcommand{\sfrac}[2]{\mbox{$\frac{#1}{#2}$}}
\newcommand{\ssup}[1] {{\scriptscriptstyle{({#1}})}}

\renewcommand{\subsection}{\secdef \subsct\sbsect}
\newcommand{\subsct}[2][default]{\refstepcounter{subsection}
\vspace{0.15cm}
{\flushleft\bf \arabic{section}.\arabic{subsection}~\bf #1  }
\nopagebreak\nopagebreak}
\newcommand{\sbsect}[1]{\vspace{0.1cm}\noindent
{\bf #1}\vspace{0.1cm}}

{\nopagebreak {\hfill\rule{2mm}{2mm}}\\ }

\newtheorem{theorem}{Theorem}[section]
\newtheorem{lemma}[theorem]{Lemma}
\newtheorem{cor}[theorem]{Corollary}
\newtheorem{prop}[theorem]{Proposition}
\newtheorem{assumption}[theorem]{Assumption}

\newtheoremstyle{thm}{1.5ex}{1.5ex}{\itshape\rmfamily}{}
{\bfseries\rmfamily}{}{2ex}{}

\newtheoremstyle{rem}{1.3ex}{1.3ex}{\rmfamily}{}
{\itshape\rmfamily}{}{1.5ex}{}
\theoremstyle{rem}
\newtheorem{remark}{{\slshape\sffamily Remark}}[]

\refstepcounter{subsubsection}

\def\thebibliography#1{\section*{References}
  \list%
  {\arabic{enumi}.}
    {\settowidth\labelwidth{[#1]}\leftmargin\labelwidth
    \advance\leftmargin\labelsep
    \parsep0pt\itemsep0pt
    \usecounter{enumi}}
    \def\newblock{\hskip .11em plus .33em minus .07em}
    \sloppy                   
    \sfcode`\.=1000\relax}

\begin{document}
\title[LDP for the local times of a RWRC in a large box]
{\large Large deviations for the local times\\ \medskip of a random walk among random conductances\\ \medskip in a growing box}
\author[Wolfgang K\"onig and Tilman Wolff]{}
\maketitle
\thispagestyle{empty}
\vspace{-0.5cm}

\centerline{\sc By Wolfgang K\"onig and Tilman Wolff}
\renewcommand{\thefootnote}{}

\vspace{-0.5cm}
\centerline{\textit{Weierstrass Institute Berlin and TU Berlin}}
\vspace{0.2cm}

\begin{center}
\version
\end{center}

\begin{quote}{\small {\bf Abstract: } We derive an annealed large deviation principle (LDP) for the normalised and rescaled local times of a continuous-time random walk among random conductances (RWRC) in a time-dependent, growing box in $\Z^d$. We work in the interesting case that the conductances are positive, but may assume arbitrarily small values. Thus, the underlying picture of the principle is a joint strategy of small conductance values and large holding times of the walk. The speed and the rate function of our principle are explicit in terms of the lower tails of the conductance distribution as well as the time-dependent size of the box. 

An interesting phase transition occurs if the thickness parameter of the conductance tails exceeds a certain threshold: for thicker tails, the random walk spreads out over the entire growing box, for thinner tails it stays confined to some bounded region. In fact, in the first case, the rate function turns out to be equal to the $p$-th power of the $p$-norm of the gradient of the square root for some $p\in(\frac {2d}{d+2},2)$. This extends the Donsker-Varadhan-G\"artner rate function for the local times of Brownian motion (with deterministic environment) from $p=2$ to these values. 

As corollaries of our LDP, we derive the logarithmic asymptotics of the non-exit probability of the RWRC from the growing box, and the Lifshitz tails of the generator of the RWRC, the randomised Laplace operator.

To contrast with the annealed, not uniformly elliptic case, we also provide an LDP in the quenched setting for conductances that are bounded and bounded away from zero. The main tool here is a spectral homogenisation result, based on a quenched invariance principle for the RWRC.}
\end{quote}


\section{Introduction and main results}\label{sec-Intro}

\noindent Random motions in random media have attracted the attention of researchers for decades because of various reasons. On one hand, they exhibit various critical behaviours that strongly differ from the classical theory in non-random media, and are sometimes surprising and on the first view counter-intuitive. This makes this subject a fascinating enterprise, a source of inspiration and beautiful mathematics and an incitation for finding new ideas and arguments. On the other hand, the introduction of randomness in the  medium makes applications in many fields much more realistic and the model therefore much more valuable. For example, random impurities in glasses, random retardations of electrical currents and much more are most efficiently modeled with the background of a random medium.

In this paper, we consider a special case of what is often called {\it random walk in random environment}; in fact it is one of its most-studied continuous-time analogues, the {\it random conductance model (RCM)}, where the randomness in the medium appears via weights on the bonds. This model was recently studied a lot (and continues to do so) with stress on the long-time behaviour of the diffusing particle in that medium, the {\it random walk among random conductances (RWRC)}. People were interested in deriving laws of large numbers, central limit theorems and invariance principles \cite{SS04, FM06, M08, BP07, BD10, A+12} in both the quenched and the annealed setting, under various assumptions on the distribution of the medium. Furthermore, heat kernel estimates \cite{BBHK08} and certain aspects of anomalous behaviour of the walk \cite{BB10} and connections with trapping models \cite{BC10} were studied. See \cite{B11} for a survey on recent progress on the random conductance model with special emphasis on 
homogenisation and martingale techniques.

However, our focus is not on the long-time behaviour in the vicinity of invariance principles in the entire space, but on the clumping behaviour in given boxes. More precisely, we derive a {\it large-deviation principle (LDP)} for the local times of a RWRC caught in boxes in the annealed setting, i.e., averaging over both randomnesses. This type of question stands in the tradition of the famous pioneering large-deviation results for the occupation times of random walks and Brownian motion from the 1970s \cite{DV75, G77}. Furthermore, there are close connections with the Lifshitz tails of the generator of the random walk in the boxes.

The present paper is a continuation of our recent study \cite{KSW11}, where we consider fixed boxes, not depending on time. In the present paper, we study large boxes that increase with time. Again, in contrast to the uniformly elliptic case, which is most often studied, we work under the assumption that the conductances are positive, but can attain arbitrarily small values, and we specify their lower tails. Then the speed of the LDP is a power of the time, and the rate function turns out to be the $p$-th power of the $p$-norm of the gradient of the square root for some $p\in(\frac {2d}{d+2},2)$. The boundary case $p=2$ is the case of the Donsker-Varadhan-G\"artner LDP mentioned above. This explicit form of the rate function makes the LDP rather appealing, and the question about the minimisers contains interesting analytical questions. This rate function is the continuous version of the rate function that we introduced in \cite{KSW11}.

Like in \cite{KSW11}, the annealed asymptotics are determined by a joint strategy of the medium and the walk, in that the conductances assume very small, time-dependent values in order to help the walk to realise large holding times in the growing box. Even more interestingly, it also turns out that there is an interesting sharp transition when the tails of the conductances at zero become thin enough: the optimal strategy consists now of an even much stronger clumping behaviour; in fact the walk confines to a fixed region that does not grow with time. In both cases, we are able to say something interesting about the non-exit probability of the walk from the growing box, and this leads, via a standard device, to the identification of the Lifshitz tails of the generator of the RWRC, the randomised Laplace operator.

One of our motivations for the present work was the desire to understand the {\it parabolic Anderson model (PAM)} with the underlying diffusion taken as a RWRC, a project that we plan to attack in future. The PAM describes a random mass flow through a random potential of sinks and sources and is determined by spectral theory of the Anderson Hamiltonian \cite{GK04,KW13}. In fact, both the generator of the PAM (the {\it Anderson Hamiltonian}) and the generator of the RWRC are important examples of random operators, and their spectral properties are of high interest. The interplay between these spectral properties and the long-time behaviour of the random walk generated makes these two models particularly interesting. As the PAM possesses self-attractive forces, the description of its behaviour heavily draws on the understanding of the clumping behaviour in given boxes, i.e., on the research brought out in the present paper. 

To contrast with the annealed setting where the conductances help the RWRC by assuming extremely small values, we also provide in Section~\ref{sec-quenchedLDP} a quenched (i.e., almost surely with respect to the conductances) LDP in growing boxes in the uniformly elliptic case, where the conductances are bounded away from zero. In this case, the conductances form a homogenised environment in which the RWRC satisfies a Donsker-type invariance principle, and the rescaled local times satisfy an LDP with rate function given by the Dirichlet energy of the limiting Brownian motion.

In the remainder of this first section, we give an introduction and formulate and comment our main results. The new contributions of this paper appear in Sections~\ref{sec-LDPSgrowing} (LDPs in large boxes), \ref{sec:intro-varprobs} (non-exit probabilities and a relevant variational problems), \ref{sec-Lifshitz} (Lifshitz tails) and \ref{sec-quenchedLDP} (a quenched LDP for uniformly elliptic conductances). Section~\ref{sec:intro-pam} explains the connection with the PAM, Section~\ref{sec:heur} gives heuristics, and in Section~\ref{sec-OpenProblems} we list some interesting problems that are left open in this paper.

\subsection{Random Walk among random conductances}\label{sec:intro-rcm}

\noindent Consider the lattice $\Z^d$ with $d\geq 1$ and a family $a=(a_{xy})_{x,y\in\Z^d}$ of non-negative random variables
$a_{xy}$. We write $\Pr$ for the corresponding probability and $\langle\cdot\rangle$ for the expectation. We assume that, $\Pr$-almost surely, $a_{xy}=a_{yx}$ for all $x,y\in\Z^d$ and $a_{xy}=0$ unless $x\sim y$, that is, unless $x$ and $y$ are nearest neighbours in the lattice. Hence, we attach to any bond on the lattice a positive random weight, and the bonds are undirected. We also sometimes write $a_{x,y}$ instead of $a_{xy}$. This model is often referred to as the {\it random conductance model (RCM)}. The most important object throughout this work will be the associated discrete Laplacian
\begin{equation}\label{eqn:random-laplace-def}
\Delta^a = \nabla^\ast A(x)\nabla,\qquad\mbox{where }
\big(A(x)\big)_{ij}=\delta_{ij}a_{x,x+e_i},\qquad x\in\Z^d,i,j\in\{1,\ldots,d\},
\end{equation}
$e_i$ is the $i$-th unit vector (with $1$ in the $i$-th component and zero everywhere else) in the lattice and $\delta_{ij}$ is the Kronecker delta. On functions $f\colon\Z^d\to\R$, the random Laplacian acts like
\begin{equation}
\Delta^a f(x) = \sum_{y\in\Z^d\colon y\sim x}a_{xy}[f(y)-f(x)].
\end{equation}
For $e\in\Ncal=\{e_1,\ldots,e_d\}$, the set of unit vectors in the lattice, we introduce $a(x,e)$ as a shortcut for $a_{x,x+e}$. We assume that the conductances are independent and identically distributed, that is, 
\begin{equation}
\big(a(x,e)\big)_{x\in\Z^d,e\in\Ncal}
\end{equation}
is an i.i.d.~family of random variables.

The operator $\Delta^a$ is symmetric and generates the continuous-time random walk $(X_t)_{t\in[0,\infty)}$ in $\Z^d$, the {\it random walk among random conductances (RWRC)}. This process starts at $x\in\Z^d$ under $\P_x^a$ and evolves as follows. When located at $y$, it waits an exponential random time with parameter $\pi_y=\sum_{z\in\Z^d\colon z\sim y}a_{y,z}$, i.e., with expectation $1/\pi_y$, and then jumps to a neighbouring site $z'$ with probability $a_{y,z'}/\pi_y$. We write $\E_x^a$ for expectation w.r.t.\ $\P_x^a$.

\subsection{Large deviations for local times in boxes}\label{sec-LDPs}

\noindent The main object of our study is the family of {\it local times} of the walk,
\begin{equation}
\ell_t(z)=\int_{0}^t\delta_{X_s}(z)\,\d s,\qquad z\in\Z^d,t>0,
\end{equation}
which register the amount of time that the walker spends in $z$ by time $t$. More precisely, we are interested in {\it large-deviation principles (LDPs)} for $\frac 1t\ell_t$ as $t\to\infty$, conditional on not leaving a given bounded region $B\subset\Z^d$. For a given choice of the conductances $a$, one of the main statements in that direction was provided by Donsker and Varadhan \cite{DV75} and G\"artner \cite{G77}.

\begin{theorem}[Donsker-Varadhan-G\"artner LDP on a finite region]\label{thm:dvg-ldp}
Fix a bounded set $B\in\Z^d$ containing 0 and a con\-duc\-tance configuration $a= (a_{xy})_{x,y\in\Z^d}$. Then, under the measures $\P_0^a(\,\cdot\,\vert\supp(\ell_t)\subset B)$, the normalised local times $\frac 1t\ell_t$ satisfy a large deviation principle on the space 
$$
\Mcal=\{g^2\colon g\in\ell^2(\Z^d),\supp (g)\subset B,\Vert g\Vert_2=1\}
$$ 
of probability measures on $B$ with scale $t$ and rate function $I^{\ssup{\rm d}}_{a,0}=I^{\ssup{\rm d}}_a-\inf_\Mcal I^{\ssup{\rm d}}_a$, where
\begin{equation}\label{eqn:dvg-ldp}
I^{\ssup{\rm d}}_a(g^2) =\sum_{e\in\Ncal}\sum_{z\in\Z^d}a_{z,z+e}\big[g(z+e)-g(z)\big]^2,\qquad g^2\in\Mcal.
\end{equation}
\end{theorem}

Here, $\Vert\cdot\Vert_2$ denotes the norm in $\ell^2(\Z^d)$, and the superscript d highlights that $B$ is a  \emph{discrete} space. Note that the terms in the sum on the right-hand side of \eqref{eqn:dvg-ldp} are non-zero only if either $z\in B$ or $z+e\in B$, that is, we are looking at a finite sum. More verbosely, the LDP says that the level sets $\{g^2\in\Mcal\colon I^{\ssup{\rm d}}_a(g^2)\leq s\}$ for $s\geq 0$ are compact, and that
\begin{eqnarray}\label{eqn:dvg-ldp-upperlower}
\liminf_{t\to\infty}\frac 1t\log\P^a_0(\ell_t\in \Ocal,\supp(\ell_t)\subset B)&\geq&-\inf_{g^2\in \Ocal}I^{\ssup{\rm d}}_a(g^2),\qquad\text{for }\Ocal\subset\Mcal\text{ open,}\\
\limsup_{t\to\infty}\frac 1t\log\P^a_0(\ell_t\in \Ccal,\supp(\ell_t)\subset B)&\leq&-\inf_{g^2\in \Ccal}I^{\ssup{\rm d}}_a(g^2),\qquad\text{for }\Ccal\subset\Mcal\text{ closed.}
\end{eqnarray}
Theorem~\ref{thm:dvg-ldp} is a {\it quenched} result, as the conductances are kept fixed. There is no interesting effect coming from the randomness of the conductances, as the number of involved random variables is finite and fixed. 

For the {\it annealed} regime, i.e., when also averaging over the conductances, there is an interesting question that arises. Under what assumptions on the environment is the annealed behaviour on a different scale than the quenched one?  Is it possible that the conductances \lq help\rq\ the walker to spend much time in $B$ by attaining very small $t$-dependent values, which slow down the movement and increase the holding times? Consequently, there would be an interplay, a compromise, between the medium and the motion. This happens in the case where the conductances are positive, but can assume arbitrarily small values. More precisely, we make the following assumption on the lower tails of the conductances.

\begin{assumption}\label{assumption:tails}
For any $x\sim y\in\Z^d$,
\begin{equation}
\Pr(a_{x,y}>0)=1\qquad\mbox{and}\qquad{\rm essinf} (a_{x,y})=0.
\end{equation}
Moreover, there exist positive parameters $\eta$ and $D$ such that, for any $x\sim y\in\Z^d$,
\begin{equation}\label{eqn:tails-behaviour}
\log \Pr (a_{x,y}\leq \varepsilon)\sim -D\eps^{-\eta}\quad\text{ as }\eps\searrow 0.
\end{equation}
\end{assumption}

The parameter $\eta$ measures the thickness of the tails at zero; the two extreme cases correspond to conductances bounded away from zero ($\eta=\infty)$ and conductances that might be zero as well ($\eta=0$). Under Assumption~\ref{assumption:tails}, the annealed asymptotic behavior of the normalised local times is indeed on a smaller scale than $t$. In our recent paper \cite{KSW11}, we obtained the following result.

\begin{theorem}[Annealed LDP, finite region]\label{thm:ldp-finite}
Suppose that Assumption~\ref{assumption:tails} holds. Then, under the annealed measures $\langle\P_0^a(\,\cdot\,\vert\supp(\ell_t)\subset B)\rangle$, the normalised local times $\frac 1t\ell_t$ satisfy a large deviation principle on the space $\Mcal$ with scale $t^{\frac{\eta}{\eta+1}}$ and rate function $J^{\ssup{\rm d}}_0=J^{\ssup{\rm d}}-\inf_\Mcal J^{\ssup{\rm d}}$, where
\begin{equation}
J^{\ssup{\rm d}}(g^2) =K_{\eta,D}\sum_{e\in\Ncal}\sum_{z\in\Z^d}\big\vert g(z+e)-g(z)\big\vert^{\frac{2\eta}{1+\eta}},\qquad g^2\in\Mcal.
\end{equation}
Here, $K_{\eta,D}=\big(1+1/\eta\big)(D\eta)^{1/(1+\eta)}$.
\end{theorem}

\subsection{LDPs in growing boxes}\label{sec-LDPSgrowing}

\noindent Now we come to the main purpose of the present paper: we extend the annealed LDP of Theorem~\ref{thm:ldp-finite} to a region $B$ that depends on time $t$ and tends to $\Z^d$. Our main motivation for this problem stems from the wish to understand a version of the parabolic Anderson model (PAM) where the underlying diffusion is itself taken random as the random conductance model; see Section~\ref{sec:intro-pam} below.

Consider a spatial scaling function $\alpha_t\in(1,\infty)$ with $1\ll\alpha_t\ll t^{1/2}$ and replace $B$ by a time-dependent, growing set $B_t=\alpha_t G\cap\Z^d$, where we fix $G\subset\R^d$ as an open, connected and bounded set containing the origin and having a sufficiently regular boundary. In order to properly incorporate the $t$-dependence of the set $B_t$, we consider the normalised and rescaled version $L_t$ of $\ell_t$, given by
\begin{equation}\label{eqn:rescaled_local_times}
L_t(x):=\frac{\alpha_t^d}{t}\ell_t(\lfloor\alpha_tx\rfloor),\qquad x\in \R^d,t>0.
\end{equation}
Observe that $L_t$ is an $L^1$-normalised random step function on $\R^d$, having support in $G$ on the event $\{\supp(\ell_t)\subset \alpha_t G\}$. Hence, $L_t$ is a member of the set 
$$
\Fcal=\{f^2\,\colon\, f\in L^2(G),\,\Vert f\Vert_2=1\},
$$ 
which we equip with the weak topology of integrals against bounded continuous functions $G\to\R$. In the simple case of constant non-random conductances $a_{xy}\equiv 1$, i.e., simple random walk, it is already known that $L_t$ conditioned on the event $\{\supp(\ell_t)\subset \alpha_t G\}$ satisfies a large deviation principle on $\Fcal$ with scale $t\alpha_t^{-2}$ and rate function $I^{\ssup{\rm c}}_0=I^{\ssup{\rm c}}-\inf_\Fcal I^{\ssup{\rm c}}$, where 
\begin{equation}\label{eqn:rate-function-brownian}
I^{\ssup{\rm c}}(f^2)=\begin{cases}
\sum_{i=1}^d\int_G \big(\partial_if(y)\big)^2\,\d y=\|\nabla f\|_2^2,&\quad f\in H_0^1(G),\\
\infty,&\quad\text{otherwise,}
       \end{cases}
\end{equation}
see e.g.~\cite{GKS07}. Here, the superscript c stands for \emph{continuous}, as the local times have rescaled to a continuous object. The additional factor of $\alpha_t^{-2}$ in the scale results from the transition from squares of differences (that occur in the Donsker-Varadhan-G\"artner rate function) to squares of derivatives in the rate function above. This also reflects the natural scaling behavior of the Laplacian, and a simple argument involving the central limit theorem easily shows that $t\alpha_t^{-2}$ is the exponential scale of the non-exit probability from a box with radius $\alpha_t$ up to time $t$.

Let us turn to annealed asymptotics in the presence of random conductances. We now establish a continuous analog to Theorem~\ref{thm:ldp-finite}. Introduce a new scale function $\gamma$ by 
$$
\gamma_t=t^{\frac\eta{1+\eta}}\alpha_t^{\frac{d-2\eta}{1+\eta}}.
$$ 
A continuous analog to the rate function in Theorem~\ref{thm:ldp-finite} is given by $J^{\ssup{\rm c}}_0=J^{\ssup{\rm c}}-\inf_{\Fcal}J^{\ssup{\rm c}}$, where 
\begin{equation}\label{ratefct}
J^{\ssup{\rm c}}(f^2)=\begin{cases}
K_{\eta,D}\sum_{i=1}^d\int_G\big\vert\partial_i f(y)\big\vert^{\frac{2\eta}{1+\eta}}\,\d y=K_{\eta,D}\sum_{i=1}^d\|\partial_i f\|_p^p,&\mbox{if }f\in H_0^1(G),\\
\infty,&\text{otherwise,}\end{cases}
\end{equation} 
where $p=\frac{2\eta}{1+\eta}\in(0,2)$, and $K_{\eta,D}$ is as in Theorem~\ref{thm:ldp-finite}. (Note that there is no standard notation for this in terms of $\nabla f$.) It turns out that $J^{\ssup{\rm c}}$ has compact level sets in the case $\eta>d/2$ only. This corresponds to conductances the tails of which at zero are not too thick. In the converse case, we thus cannot hope for a full LDP to hold. Let us for that reason consider the case $\eta>d/2$ first. Recall that $G$ is a bounded open set containing the origin with regular boundary.

\begin{theorem}[Annealed asymptotics, time-dependent region]\label{thm:ldp-growing-good}
Suppose that Assumption~\ref{assumption:tails} holds, and assume that $\eta>d/2$. In case $d=1$, suppose that $\eta\geq 1$. Furthermore, assume that the conductances are bounded almost surely and that $a_{xy}\1\{a_{xy}\leq \eps\}$ possesses, for some $\eps>0$, a density that is non-decreasing. Pick a scale function $(\alpha_t)_{t>0}$ such that  $1\ll\alpha_t^{d+2}\ll t(\log t)^{-(1+\eta)/\eta}$. 

Then the distributions of $L_t$ under the conditional annealed measures $\langle\P_0^a(\,\cdot\,|\,\supp(\ell_t)\subset \alpha_t G)\rangle$ satisfies a large-deviation principle on $\Fcal$ with good rate function $J^{\ssup{\rm c}}_0$.
\end{theorem}

More explicitly, Theorem~\ref{thm:ldp-growing-good} says that $J^{\ssup{\rm c}}_0$ is has compact level sets, and 
\begin{eqnarray}
\liminf_{t\to\infty}\frac 1{\gamma_t}\log\langle\P^a_0(L_t\in \Ocal,\supp(\ell_t)\subset \alpha_tG)\rangle&\geq&-K_{\eta,D}\chi^{\ssup  {\rm c}}(G,\Ocal),\qquad\text{for }\Ocal\subset\Fcal\text{ open,}\label{eqn:dvg-ldp-lowerLarge}\\
\limsup_{t\to\infty}\frac 1{\gamma_t}\log\langle\P^a_0(L_t\in \Ccal,\supp(\ell_t)\subset \alpha_tG)\rangle&\leq&-K_{\eta,D}\chi^{\ssup  {\rm c}}(G,\Ccal),\qquad\text{for }\Ccal\subset\Fcal\text{ closed,}\label{eqn:dvg-ldp-upperLarge}
\end{eqnarray}
where
\begin{equation}\label{eqn:varprob-continuousAcal}
\chi^{\ssup{\rm c}}(G,\Acal)=\inf\Big\{\sum_{i=1}^d\int_G\big\vert\partial_i f(y)\big\vert^{\frac{2\eta}{1+\eta}}\,\d y\,\colon\,f\in H_0^1(G),\,\Vert f\Vert_2=1, f^2\in\Acal\Big\}.
\end{equation}

A heuristic explanation of Theorem~\ref{thm:ldp-growing-good} is in Section~\ref{sec:heur}. The proof is in Section~\ref{sec:ldp-growing-proofs}.  The technical assumption on the existence of an increasing density of small conductances will be used in the proof of the lower bound, where we will confine the conductances very strongly. The technical assumption on the boundedness and the additional logarithmic term in the upper bound for $\alpha_t$ will help us to make the proof of the upper bound less cumbersome. 

There is no reason to expect that the rate function $J^{\ssup{\rm c}}_0$ is convex. In \eqref{Jcident} we give an alternative formula for $J^{\ssup{\rm c}}_0$, but also this gives no hint at convexity, since the min-max-formula for interchange of infimum and supremum \cite[p.~151]{DZ98} cannot be applied, unlike in \cite{CGZ00} at the end of Section 3. Rather we presume that $J^{\ssup{\rm c}}_0$ is not convex. See \cite[Prop.~4]{K00} for a proof of convexity in the case $p\geq 2$.

As we already mentioned in connection with Theorem~\ref{thm:ldp-finite}, and as we will explain in detail in Section~\ref{sec:heur}, the main contribution of the conductances to the LDP is to assume very small values, in order to make it easier for the walk to stay in the set $\alpha_t G$ for $t$ time units; this is a large-deviation event by the assumption $1\ll\alpha_t\ll t^{\frac 1{d+2}}$. By the assumption $\eta>d/2$, the probabilistic cost for this contribution is small enough that it can be performed all over the growing set $\alpha_t G\cap\Z^d$, as the cost for assuming small values is not too high. We will see in the next section that $d/2$ is precisely the threshold for $\eta$ for this to happen.

\subsection{Non-exit probabilities, variational formulas, and the case $\boldsymbol{\eta\leq d/2}$}\label{sec:intro-varprobs}

\noindent Let us look at non-exit probabilities and find two independent arguments, a probabilistic and an analytic one, for the existence of an interesting phase transition, as $\eta$ traverses $d/2$.

As a corollary of Theorem~\ref{thm:ldp-finite}, we pointed out in \cite{KSW11} that the non-exit probability from the finite region $B$ satisfies 
\begin{equation}\label{eqn:nonexit-finite}
\log\langle\P^a_0\big(\supp(\ell_t)\subset B\big)\rangle\sim -t^{\frac\eta{1+\eta}}K_{\eta,D}\chi^{\ssup{\rm d}}(B),
\end{equation}
where
\begin{equation}\label{eqn:varprob-discrete}
\chi^{\ssup{\rm d}}(B)=\inf\Big\{\sum_{e\in\Ncal}\sum_{z\in\Z^d}\big\vert g(z+e)-g(z)\big\vert^{\frac{2\eta}{1+\eta}}\,\colon\,g\in \ell^2(\Z^d),\,\supp (g)\subset B,\,\Vert g\Vert_2=1\Big\}.
\end{equation}
In the same way, we obtain as a corollary from Theorem~\ref{thm:ldp-growing-good} that, in the case $\eta>d/2$,
\begin{equation}\label{eqn:nonexit-growing}
\log\langle\P^a_0\big(\supp(\ell_t)\subset\alpha_tG\big)\rangle\sim -t^{\frac\eta{1+\eta}}\alpha_t^{\frac{d-2\eta}{1+\eta}}K_{\eta,D}\chi^{\ssup{\rm c}}(G),
\end{equation}
where $\chi^{\ssup{\rm c}}(G)=\chi^{\ssup{\rm c}}(G,\Fcal)$ is the continuous version of $\chi^{\ssup{\rm d}}(B)$; see \eqref{eqn:varprob-continuousAcal}.

However, in the case $\eta\leq d/2$, \eqref{eqn:nonexit-growing} is awkward, since the left-hand side is obviously non-decreasing in $\alpha_t$, but the right-hand side is non-increasing. This suggests that $\chi^{\ssup{\rm c}}(G)=0$ in that case. The following result shows that the non-exit probability is in fact on a slower scale.

\begin{theorem}\label{thm:ldp-growing-bad}
Suppose $1\ll\alpha_t\ll t^{\frac \eta{d(\eta+1)}}$ and that Assumption~\ref{assumption:tails} holds. In addition, assume that $\eta\leq d/2$. Then,
\begin{itemize}
\item[(i)] The level sets of $J^{\ssup{\rm c}}$ are not closed and in particular not compact,
\item[(ii)] for all finite and connected sets $B\subset\Z^d$ containing the origin,
\begin{equation}\label{eqn:ldp-growing-bad-lbound}
\liminf_{t\to\infty}t^{-\frac{\eta}{\eta+1}}\log\langle\P^a_0\big(\supp(\ell_t)\subset\alpha_tG\big)\rangle\geq -K_{\eta,D}\chi^{\ssup{\rm d}}(B),
\end{equation}
\item[(iii)]
\begin{equation}\label{eqn:ldp-growing-bad-ubound}
\limsup_{t\to\infty}t^{-\frac{\eta}{\eta+1}}\log\langle\P^a_0\big(\supp(\ell_t)\subset\alpha_tG\big)\rangle\leq -K_{\eta,D}\chi^{\ssup{\rm d}}(\Z^d).
\end{equation}
\end{itemize}
In the case $\eta=d/2$ we have the corresponding lower bound
\begin{equation}\label{eqn:ldp-growing-bad-lbound-exact}
\liminf_{t\to\infty}t^{-\frac{\eta}{\eta+1}}\log\langle\P^a_0\big(\supp(\ell_t)\subset\alpha_tG\big)\rangle\geq -K_{\eta,D}\chi^{\ssup{\rm d}}(\Z^d).
\end{equation}
\end{theorem}

Hence, the leading-order logarithmic asymptotics of the non-exit probability do not depend on the set $G\subset\R^d$ nor on the scale function $\alpha_t$. The proof of Theorem~\ref{thm:ldp-growing-bad} is in Section~\ref{sec-proofThmBad}. We will see in Section~\ref{sec:heur} below that the heuristics for the LDP of Theorem~\ref{thm:ldp-growing-good} also apply for the case $\eta\leq d/2$ of Theorem~\ref{thm:ldp-growing-bad}. Its Assertion (i) gives a first reason why nevertheless the LDP does not hold true. Assertion (iii) gives another one: Except for the special case $\eta=d/2$, we clearly have 
$$
\gamma_t=t^{\frac\eta{1+\eta}}\alpha_t^{\frac{d-2\eta}{1+\eta}}\ll t^{\frac{\eta}{1+\eta}}.
$$ 
This means that the non-exit probability is on a slower (i.e., probabilistically less costly) scale than the one the LDP in Theorem \ref{thm:ldp-growing-good} would imply.

A heuristic explanation is the fact that $\eta\leq d/2$ corresponds to a high probabilistic cost for very small conductances. Therefore, the non-exit probability is governed by the event where conductances are very small only on a bounded number of sites, or at the most on a set of sites much smaller that $B_t$, in contrast to the event where conductances are small \emph{everywhere} which would lead to the scale $\gamma_t$. The random walk is then slowed down so much that it does not even leave the smaller set. Theorem \ref{thm:ldp-growing-bad} shows that this is exactly the behavior that governs annealed asymptotics, at least those of non-exit probabilities, in the case $\eta\leq d/2$. 

Combining the results of Theorems \ref{thm:ldp-growing-good} and \ref{thm:ldp-growing-bad}, we would like to remark that the scale of the non-exit probabilities is decreasing in $\eta$ across all values $\eta>0$, since under the restriction that $1\ll\alpha_t\ll t^{\frac \eta{d(\eta+1)}}$,
$$
\gamma_t=t^{\frac\eta{1+\eta}}\alpha_t^{\frac{d-2\eta}{1+\eta}}\gg t^{\frac{\eta^\ast}{1+\eta^\ast}}\qquad\mbox{for any }\eta>\eta^\ast=d/2.
$$

The different behaviours in the two regimes are also reflected by analytic properties of the arising variational problems, as we will see now. In fact, for $\eta>d/2$, the continuous variational problems are well-behaved and admit standard compactness arguments, but not the discrete ones, and {\it vice versa}. Recall that $\chi^{\ssup{\rm c}}(G)$ equals $\chi^{\ssup{\rm c}}(G,\Fcal)$ defined in \eqref{eqn:varprob-continuousAcal}.

\begin{prop}\label{prop:varprobs}
\begin{itemize}
\item[(i)] Assume that $\eta>d/2$. Then,
\begin{itemize}
\item[$\bullet$] $\chi^{\ssup{\rm c}}(G)>0$, and the continuous variational problem in \eqref{eqn:varprob-continuousAcal} for $\Acal=\Fcal$  possesses at least one minimiser. In the case $d=1$, we need to make the additional assumption that $\eta\geq 1$.
\item[$\bullet$] $\chi^{\ssup{\rm d}}(\Z^d)=0$ and the discrete variational problem in \eqref{eqn:varprob-discrete} (with $B=\Z^d$)  has no minimiser.
\end{itemize}
\item[(ii)] Assume that $\eta\leq d/2$. Then,
\begin{itemize}
\item[$\bullet$] $\chi^{\ssup{\rm c}}(G)=0$ and the continuous variational problem in \eqref{eqn:varprob-continuousAcal}  for $\Acal=\Fcal$ has no minimiser.
\item[$\bullet$] $\chi^{\ssup{\rm d}}(\Z^d)>0$ if and only if $d>1$.
\end{itemize}
\end{itemize}
\end{prop}

The proof of Proposition~\ref{prop:varprobs} is in Section~\ref{sec:varprobs}.

 \subsection{Lifshitz tails for the principal eigenvalue}\label{sec-Lifshitz}
 
\noindent Let us denote by $\lambda^a(B)$ the bottom of the spectrum of $-\Delta^a$ in the connected set $B\subset\Z^d$ with Dirichlet (i.e., zero) boundary condition. Using the abbreviation $a(x,e)=a_{x,x+e}$, the well-known Rayleigh-Ritz formula reads
\begin{equation}\label{eqn:RR_LapOmega}
\lambda^a(B)=\inf\Big\{\sum_{z\in\Z^d}\sum_{e\in\Ncal}a(z,e)(g(z+e)-g(z))^2\,\colon\,g\in \ell^2(\Z^d),\Vert g\Vert_2=1,\,\supp(g)\subset B\Big\}.
 \end{equation}
Under Assumption~\ref{assumption:tails}, $\lambda^a(B)$ is a positive random variable with essential infimum equal to zero, and its tails at zero are of high interest from the viewpoint of Lifshitz tails of the random operator $-\Delta^a$. In \cite{KSW11}, we proved as a corollary of Theorem~\ref{thm:ldp-finite} that, for $B$ a fixed bounded set, the Lifshitz tails are given by
\begin{equation}\label{LifshitzB}
\lim_{\eps\downarrow 0}\eps^{\eta}\log\Pr(\lambda^a(B)\leq \eps)=-D\chi^{\ssup{\rm d}}(B)^{\eta+1}.
\end{equation}
 
Now, Theorem~\ref{thm:ldp-growing-good} also yields the analogous corollary for the Lifshitz tails in the $t$-dependent set $B=B_t=\alpha_t G\cap\Z^d$ with $G\subset\R^d$ as in  Theorem~\ref{thm:ldp-growing-good}. For simplicity, we restrict to the case where $\alpha_t$ is a power of $t$.

\begin{cor}\label{Cor-Lifshitz}
Suppose that the assumptions of Theorem~\ref{thm:ldp-growing-good} are satisfied; in particular we assume that $\eta>d/2$. Furthermore, assume that $\alpha_t=t^{s/(d-2\eta)}$ for some $s\in(0,\frac{d-2\eta}{d+2})$. Then
\begin{equation}\label{LifshitzC}
\lim_{\eps\downarrow 0}\eps^{\eta+s}\log\Pr(\lambda^a(\alpha_t G\cap\Z^d)\leq \eps^{1-s})=-\Big(\frac 1\eta \chi^{\ssup{\rm c}}(G)\Big)^{\eta+1}(1-s)^{1-s}(\eta+s)^{\eta+s}.
\end{equation}
\end{cor}

Certainly, from Theorem~\ref{thm:ldp-growing-bad}, one can deduce an analogous statement also in the case $\eta=d/2$, but our precision in the case $\eta<d/2$ is not high enough for deriving Lifshitz tails.

The proof of Corollary~\ref{Cor-Lifshitz} is a variant of the proof of \eqref{LifshitzB} in \cite{KSW11}. It uses the fact that
$$
\log\big\langle\e^{t\lambda^a(\alpha_t G\cap\Z^d)}\big\rangle \sim \log\big\langle\P^a_0\big(\supp(\ell_t)\subset\alpha_tG\big)\big\rangle,\qquad t\to\infty,
$$
which is easy to show by standard arguments (also using that we indeed prove the upper bound in \eqref{eqn:dvg-ldp-upperLarge} for any starting point uniformly). Using now the asymptotics from \eqref{eqn:nonexit-growing} and applying de Bruijn's exponential Tauberian theorem \cite[Theorem 4.12.9]{BGT89} yields the assertion.

\subsection{A quenched LDP for uniformly elliptic conductances}\label{sec-quenchedLDP}

To contrast with the main topic of the present paper (the annealed setting for conductances whose essential infimum is zero) we give now a result in the quenched setting (i.e., with probability one with respect to the conductances) for conductances that are bounded and bounded away from zero, in which case the environment is called {\it uniformly elliptic}. Again, we consider an open bounded set $G$ that contains the origin and a scale function $\alpha_t\gg1$ and consider the RWRC in the growing box $B_t=\alpha_t G \cap\Z^d$. In this case, the conductances cannot have any tendency to assume extreme values, but will form a more or less homogeneous environment, and the random walk will behave qualitatively like in the LDP of \cite{GKS07} (mentioned around \eqref{eqn:rate-function-brownian}) in this homogenised environment.  Accordingly, we will be using techniques from the theory of stochastic homogenisation, and we will rely on a quenched functional central limit theorem. The latter states that the RWRC, 
rescaled in the standard way as in Donsker's invariance principle, converges in probability towards a Brownian motion with covariance matrix $c_\eff\Id$, see \cite{A+12}, e.g. The constant $c_\eff>0$ is called \emph{effective diffusion constant} or \emph{effective conductivity} and depends in a rather complex way upon the conductance distribution. 

For simplicity, we restrict to the case where $G$ is a cube.

\begin{theorem}[Quenched LDP for uniformly elliptic conductances]\label{thm:ldp-quenched}
Assume that $\lambda\le a_{xy}\le\frac1\lambda$ almost surely, for some $\lambda\in(0,1)$. Moreover, assume that $G=(0,1)^d$ is the open unit cube. Then, $\Pr$-almost surely, the rescaled local times $L_t$ under $\P_0^a\big(\cdot\vert\supp(\ell_t)\subset \alpha_t G\big)$ satisfy a large deviation principle on $\Fcal$ with scale $t\alpha_t^{-2}$ and rate function $c_\eff I^{\ssup{\rm  c}}_0$ defined in \eqref{eqn:rate-function-brownian}. 
\end{theorem}

We will prove this theorem in Section~\ref{sec:localtimes-quenched}. The proof relies on a spectral homogenisation result from \cite{BD03}, which states that the eigenvalues and eigenfunctions of the rescaled discrete random Laplacian on $B_t=\alpha_t G \cap\Z^d$ behave on the large scale like those of  the continuous counterpart $c_\eff\Delta$ on $G$. We mention that this assertion has been proved only for i.i.d.~conductances yet.

\subsection{Relevance for the parabolic Anderson model}\label{sec:intro-pam}

\noindent As we mentioned above, one of our main motivations for the present study stems from the interest in understanding the \emph{parabolic Anderson model (PAM)} with additional randomness in the diffusivity given by random conductances. The usual PAM is the solution to the heat equation on $\Z^d$ with random potential, see \cite{GK04} and \cite{KW13} and the references therein for more background. Consider $u\colon[0,\infty]\times \Z^d$ solving the Cauchy problem
\begin{align*}
\begin{cases}
\frac{\partial}{\partial t}u(t,z)=\Delta u(t,z)+\xi(z)u(t,z),&(t,z)\in[0,\infty]\times \Z^d,\\
u(0,z)=\delta_0(z),&z\in\Z^d,
\end{cases}
\end{align*}
where $\xi=(\xi(z))_{z\in\Z^d}$ is a real-valued random potential. For simplicity, we assume that $\xi$ is an i.i.d.~collection of random variables. The solution $u$ can be represented in terms of the Feynman-Kac formula as an expectation over a continuous-time simple random walk with generator $\Delta$. Its total mass $U(t)=\sum_{z\in\Z^d}u(t,z)$ can then be written as
$$
U(t)=\E_0\Big(\exp\Big\{\sum_{z\in\Z^d} \xi(z)\ell_t(z)\Big\}\Big).
$$
From here, one can already suspect that one of the keys in understanding, or at least proving, the large-$t$ behaviour would be a good control on the large deviations of the local times of the walks, and in many research papers this indeed turned out to be decisive. This gets even more convincing when we look at the expectation of $U(t)$ with respect to $\xi$, which equals, as one can see from an elementary calculation,
\begin{equation}\label{Uident}
\E_0\Big(\exp\Big\{\sum_{z\in\Z^d} H(\ell_t(z))\Big\}\Big),
\end{equation}
where $H(\ell)=\log {\tt E}(\e^{\ell \xi(0)})$ denotes the logarithm of the moment generating function. Since $H$ is a convex function, this term has a self-attracting effect on the random walk, hence the description of the large-$t$ behaviour requires a deep understanding of the asymptotic behaviour of the local times in boxes on length scales that are much smaller than the scale of the central limit theorem, i.e., having radii $\ll  \sqrt t$. The size of the relevant box depends on the large-$\ell$ asymptotics of $H(\ell)$. An example is the case where $\xi(0)$ has double exponential tails, where the relevant box turns out not to depend on $t$ \cite{GM98}. For bounded potentials, it has a radius that diverges like a power $\leq 1/(d+2)$ of $t$ \cite{BK01}. 

It is of interest to introduce randomness also in the diffusivity, i.e., to replace the Laplace operator $\Delta$ by the randomised one, $\Delta^a$, and the study of this model is our future goal. From the above, it is clear that all we have to do for identifying the expected total mass is to replace $\E_0$ in \eqref{Uident} by $\E_0^a$, i.e., the simple random walk by the RWRC. Hence, the large-deviation principles of the present paper will be an indispensable help for this future task.

\subsection{Heuristic derivation of Theorem~\ref{thm:ldp-growing-good}}\label{sec:heur} 

\noindent Let us present a heuristic derivation of the LDP of Theorem \ref{thm:ldp-growing-good}, will serve also as an outline for the proof of the lower bound in Theorem \ref{thm:ldp-growing-good}, and it introduces some notation that will be frequently used later. Let us fix any $\eta\in(0,\infty)$; the following does not depend on whether $\eta$ is smaller or larger than $d/2$. We intend to find the asymptotics for the annealed probability of the event $\{L_t\approx f^2\}$ for any $f^2\in\Fcal$, and we keep in mind that this event is to be interpreted as $\{L_t\approx f^2,\supp(\ell_t)\subset\alpha_t G\}$.

The main idea is to find the conductance profile contributing optimally to the probability of the event, and to apply an LDP for the local times given this particular conductance profile. As opposed to the finite region case, the optimal realisation of conductances will depend on time. Let us therefore consider the rescaled conductance field 
\begin{equation}\label{omegatdef}
a_t(y,e)=\beta_t a(\lfloor\alpha_t y\rfloor,e),\qquad e\in\Ncal, y\in G,
\end{equation}
and the scale function $\beta_t\gg 1$ will be chosen along the way (recall our convention $a(z,e)=a_{z,z+e}$ for $z\in\Z^d$ and $e\in\Ncal$). We consider the event that $a_t$ resembles a given function $\varphi\colon G\times\Ncal\to(0,\infty)$, i.e., we approximate
\begin{equation}\label{eqn:ldp-growing-heur0}
\langle\P^a(L_t\approx f^2)\rangle\approx\langle\P^a(L_t\approx f^2)\1\{a_t\approx \varphi\mbox{ on }G\times \Ncal\}\rangle
\end{equation}
for some optimal conductance shape $\varphi$. Let us first calculate the exponential decay rate of the probability of $\{a_t\approx \varphi\mbox{ on }G\times \Ncal\}$. Based on Assumption \ref{assumption:tails}, we obtain
\begin{align}\label{eqn:ldp-growing-heur1}
\log\Pr(a_t\approx \varphi\mbox{ on }G\times\Ncal)
&\approx\log\Big(\prod_{e\in\Ncal}\prod_{z\in\alpha_tG\cap\Z^d}\Pr\big(a(z,e)\approx\beta_t^{-1}\varphi(z/\alpha_t,e)\big)\notag\Big)\\
&\approx-D\beta_t^\eta\sum_{e\in\Ncal}\sum_{z\in\alpha_tG\cap\Z^d}\varphi(z/\alpha_t,e)^{-\eta}\notag\\
&\approx-\beta_t^\eta\alpha_t^d\sum_{e\in\Ncal}D\int_G\varphi(y,e)^{-\eta}\,\d y.
\end{align}
(We will present a more rigorous version of this in Lemma \ref{lem:lbound-environment}.) On the other hand, we may evaluate the $\P^a$-probability of $\{L_t\approx f^2\}$  on the event $\{a_t\approx \varphi\mbox{ on }G\times \Ncal\}$ in terms of a rescaled version of the famous Donsker-Varadhan-G\"artner large deviation principle. In analogy with the large deviation principle for $L_t$ mentioned in Section \ref{sec:intro-rcm} for the simple random walk case,
\begin{equation}\label{eqn:ldp-growing-heur2}
\P^{a}(L_t\approx f^2)\approx \exp\Big(-\frac{t}{\alpha_t^2\beta_t}\sum_{e\in\Ncal}\int_G\varphi(y,e)\big(\partial_ef\big)^2(y)\,\d y\Big)
\end{equation}
on the event where $\{a_t\approx \varphi\mbox{ on }G\times \Ncal\}$. This is well-aligned with the rate function given in \eqref{eqn:rate-function-brownian}, and Proposition~\ref{lem:rescaled-fix-ldp} in Section \ref{sec:ldp-growing-prelim} gives an account of this in a more rigorous way. Combining the approximations in \eqref{eqn:ldp-growing-heur1} and \eqref{eqn:ldp-growing-heur2} with \eqref{eqn:ldp-growing-heur0}, we obtain
\begin{equation}\label{eqn:ldp-growing-heur3}
\log\langle\P^a(L_t\approx f^2)\rangle\approx -\frac{t}{\alpha_t^2\beta_t}\sum_{e\in\Ncal}\int_G\varphi(y,e)\big(\partial_ef\big)^2(y)\,\d y - \beta_t^\eta\alpha_t^d\sum_{e\in\Ncal}D\int_G\varphi(y,e)^{-\eta}\,\d y.
\end{equation}
The decay rate on the right-hand side is minimal if we choose $\beta_t$ such that 
\begin{equation}\label{betadef}
\frac{t}{\alpha_t^2\beta_t}=\beta_t^\eta\alpha_t^d,\qquad\mbox{i.e.,}\qquad \beta_t= \Big(\frac t{\alpha_t^{d+2}}\Big)^{\frac 1{1+\eta}}.
\end{equation} 
Note that the condition $\alpha_t\ll t^{\frac 1{d+2}}$ from Theorem \ref{thm:ldp-growing-good} ensures that $\beta_t\gg 1$. Furthermore, the optimal scale is now seen to be equal to
\begin{equation}\label{gammatdef}
\gamma_t=\frac{t}{\alpha_t^2\beta_t}=\beta_t^\eta\alpha_t^d=t^{\frac\eta{1+\eta}}\alpha_t^{\frac{d-2\eta}{1+\eta}}.
\end{equation}
The optimal shape $\varphi$ is determined by the minimisation of the sum of the two integrals on the right-hand side of \eqref{eqn:ldp-growing-heur2}. Minimizing term by term, we see that
$$
\varphi(y,e)=\arg\inf\Big\{r\big(\partial_ef(y)\big)^2+Dr^{-\eta}\,\colon\,r\in[0,\infty]\Big\},
$$
which yields
$$
\varphi(y,e)\big(\partial_ef(y)\big)^2+D\varphi(y,e)^{-\eta}=K_{\eta,D}\big|\partial_ef(y)\big|^{p},\qquad y\in G,e\in\Ncal,
$$
with $K_{\eta,D}$ as in Theorem~\ref{thm:ldp-growing-good} and $p=\frac{2\eta}{\eta+1}$. In particular, we have identified the rate function $J^{\ssup{\rm c}}$ of \eqref{ratefct} as 
\begin{equation}\label{Jcident}
J^{\ssup{\rm c}}(f^2)=\inf_{\varphi\colon G\times\Ncal\to(0,\infty)}\Big[\sum_{e\in\Ncal}\int_G\varphi(y,e)\big(\partial_ef\big)^2(y)\,\d y +\sum_{e\in\Ncal}D\int_G\varphi(y,e)^{-\eta}\,\d y\Big].
\end{equation}
This ends our heuristic explanation of the LDP in Theorem~\ref{thm:ldp-growing-good}.

\subsection{Open problems}\label{sec-OpenProblems}

The present work leaves open a number of interesting questions, both on the analytic and the probabilistic side. It is open whether or not the rate functions $J^{\ssup {\rm c}}$ and $J^{\ssup {\rm d}}$ are linked with some interesting operator on its own right, like the pseudo-$p$-Laplacian. See \cite{BK04} for the study of a problem that is closely related with the analysis of the minimiser(s) of $J^{\ssup {\rm c}}$. Another question concerns the precise behaviour of the minimisers of the formula for $\chi^{\ssup {\rm d}}(B)$ for $B\uparrow\Z^d$ in the three cases $\eta<d/2$, $\eta=d/2$ and $\eta>d/2$: do we have convergent subsequences, and does a continuous or a discrete picture arise? On the probabilistic side, it would be interesting to find methods to determine the asymptotic shape of the local times conditional on staying in $\alpha_t G$ for $\eta\leq d/2$, where we expect a discrete picture to arise. Furthermore, the methods of the present paper are not strong enough to rigorously identify the behaviour of the conductances under the annealed law, conditional on the walk not leaving the set $\alpha_t G$; also this is interesting. Moreover, the quenched setting (i.e., with probability one with respect to the conductances) is rather interesting as well; is it true that a similar picture as for the PAM arises: the random walk quickly moves to a remote small region in which the conductances create a particularly preferable environment? And lastly, of course the model that gave the main motivation of this paper remains to investigated, the PAM with diffusivity taken equal to the RWRC.

\section{The characteristic variational problems}\label{sec:varprobs}

\noindent In this section, we prove Proposition~\ref{prop:varprobs}. It follows from a couple of lemmas that we are going to state and prove. All results of this section are self-contained and do not need any probabilistic input. Nevertheless, the proof of the upper bound in Theorem~\ref{thm:ldp-growing-good} also relies on some of the results presented in this section.

%
Let us state, for future reference, a form of the Rellich-Kondrashov theorem, which  the reader may find in \cite[Theorem 8.9]{LL01}, for instance.

\begin{theorem}[Rellich-Kondrashov]\label{thm:rellich-kondrashov}
Let $1\leq p\leq\infty$ and $f,f_1,f_2,\ldots\in W_0^{1,p}(G)$ such that $f_n\to f$ weakly. Then
\begin{itemize}
\item[i)] If $p<d$, then $\Vert f_n-f\Vert_q\to 0$ for all $q$ with $1\leq q<\frac{dp}{d-p}$.

\item[ii)] If $p=d$, then $\Vert f_n-f\Vert_q\to 0$ for all $q\in(0,\infty)$.

\item[iii)] If $p>d$, then $\Vert f_n-f\Vert_\infty\to 0$.

\end{itemize}
\end{theorem}

\begin{lemma}\label{lem:varprob-cont-solvable}
If $\eta>d/2$ (in dimension $d=1$, assume in addition that $\frac{2\eta}{\eta+1}\geq 1$), then the continuous variational problem in \eqref{eqn:varprob-continuousAcal} for $\Acal=\Fcal$ has a minimiser.
\end{lemma}

\begin{proof}
Put $p=\frac{2\eta}{\eta+1}<2$ and choose a sequence $(f_n)_{n\in\N}$ in $H_0^1(G)$ with $\Vert f_n\Vert_2=1$ for all $n\in\N$ that satisfies 
$$
\lim_{n\to\infty}\sum_{e\in\Ncal}\|\partial_e f_n\|_p^p=\chi^{\ssup{\rm c}}(G).
$$ 
Clearly, the $p$-norms of all derivatives $\partial_ef_n$ with $e\in\Ncal$ must be bounded as the sequence approximates the infimum. In addition, we may estimate 
$$
\Vert f_n\Vert_p^p=\Vert f_n\1_{\{f_n>1\}}\Vert_p^p+\Vert f_n\1_{\{f_n\leq 1\}}\Vert_p^p\leq \Vert f_n\Vert_2 +\vert G\vert=1+\vert G\vert,
$$ 
which means that the sequence $(f_n)_n$ is bounded in $W^{1,p}$. Consequently, we may assume that it converges weakly towards some $f\in W^{1,p}$. We now have to check the conditions in the Rellich-Kondrashov theorem above (with the choice $q=2$) to establish strong convergence of $f_n$ in $L^2(G)$.

\emph{Case $d\geq 2$:} On the one hand, $p>\frac{2d}{d+2}$, so in particular $p\geq 1$. On the other, we have $p<2\leq d$. In order to use Theorem \ref{thm:rellich-kondrashov} i) with $q=2$, we just estimate $$\frac{dp}{d-p}=\frac{2d\eta}{d\eta+d-2\eta}>2.$$

\emph{Case $d=1$:} By the additional assumption, $p\geq 1$. Therefore, we may either use Theorem \ref{thm:rellich-kondrashov} ii) or iii).

We have now shown that $f_n\to f$ strongly in $L^2(G)$ and in particular $\Vert f\Vert_2=1$. As $\partial_e f_n\to \partial_ef$ weakly for all $e\in\Ncal$ and $L^p$-norms are lower semicontinuous with regard to the weak topology (see e.g.~\cite{LL01}, Section 2.11), we have that
$\sum_{e\in\Ncal}\Vert\partial_e f\Vert_p^p\leq\sum_{e\in\Ncal}\liminf_{n\to\infty}\Vert\partial_e f_n\Vert_p^p$, i.e., $f$ is a minimiser. This finishes the proof of Lemma \ref{lem:varprob-cont-solvable}.
\end{proof}

\begin{remark} The case $d=1$ and $p=\frac{2\eta}{\eta+1}<1$ is not accessible to the techniques above as the map $f\mapsto (\int_G f^p)^{1/p}$ is not even a seminorm if $p<1$. 
\end{remark}

In the following, we write $|\cdot|_r$ for the standard $r$-norm on $\R^d$.

\begin{lemma}\label{lem:varprob-cont-not-solvable}
If $\eta\leq d/2$, then $\chi^{\ssup{\rm c}}(G)=0$ and the continuous variational problem in \eqref{eqn:varprob-continuousAcal} for $\Acal=\Fcal$ does not have a minimiser.
\end{lemma}

\begin{proof}
It will be sufficient to show that $\chi^{\ssup{\rm c}}(G)=0$. Pick $\eps_0>0$ such that the open ball with radius $\eps_0$ around the origin is contained in $G$. The proof is separated into the cases $d=1$ and $d\geq 2$.

\emph{Case} $d=1$:
Here, we have $p\geq 2d/(d+2)=2/3$. For $r>0$, define $f_r(x)=A_r(\eps_0-\vert x\vert)^r \1_{\{\vert x\vert<\eps_0\}}$ with $A_r^2=\frac{2r+1}{2\eps_0^{2r+1}}$. We easily check that $f_r\in H_0^1(G)$, $\Vert f_r\Vert_2=1$ and $\vert f'(x)\vert=rA_r(\eps_0-\vert x\vert)^{r-1}\1_{\{\vert x\vert<\eps_0\}}$. Then,
$$
\int_G\vert f'(x)\vert^p\,\d x=2r^pA_r^p\frac 1{pr-p+1}\eps_0^{pr-p+1}\leq C r^p r^{p/2} \eps_0^{-pr} r^{-1}\eps_0^{pr}=Cr^{\frac{3p}2 -1}
$$
for some constant $C>0$. As the last term obviously vanishes for $r\to\infty$, the assertion is shown in the case $d=1$.

\emph{Case} $d\geq 2$:
We construct a family $(f_\eps)_{\eps\in(0,\eps_0)}$ of functions in $H_0^1(G)$ with $\Vert f_\eps\Vert_2=1$ and $\sum_e\|\partial_e f_\eps\|_p^p\to 0$ as $\eps\to 0$, where we recall that $p=\frac{2\eta}{1+\eta}$. Choose some $\gamma\in(d/4,d/2)$ and put 
$$
\tilde f_\eps(x)=\big(\vert x\vert_2^{-2\gamma}-\eps^{-2\gamma}\big)^{1/2}\1_{\{\vert x\vert_2<\eps\}},
$$ 
to obtain
\begin{equation*}
\Vert \tilde f\Vert_2^2=d\Omega_d\int_0^\eps [r^{-2\gamma}-\eps^{-2\gamma}] r^{d-1}\,\d r =C_1\eps^{d-2\gamma},
\end{equation*}
where $\Omega_d$ denotes the volume of the unit ball in $\R^d$, $C_1$ is some appropriate constant, and the existence of the integral above follows from $\gamma<d/2$. Choosing $A_\eps^2=C_1^{-1}\eps^{2\gamma-d}$, the functions $f_\eps=A_\eps\tilde f_\eps$ are $L^2(G)$-normed. Moreover, for $x\in\R^d$ with $\vert x\vert_2<\eps$,
\begin{align*}
\vert\nabla f_\eps(x)\vert_2^2&=\sum_{i=1}^d\Big\vert\frac\partial{\partial x_i}\Big[A_\eps\big(\vert x\vert_2^{-2\gamma}-\eps^{-2\gamma}\big)^{1/2}\Big]\Big\vert^2\\
&=A_\eps^2\sum_{i=1}^d\Big\vert\frac 1 2\big(\vert x\vert_2^{-2\gamma}-\eps^{-2\gamma}\big)^{-1/2}\cdot \gamma\vert x\vert_2^{-2\gamma-2}\cdot 2x_i\Big\vert^2\\
&=A_\eps^2\gamma^2\big(\vert x\vert_2^{-2\gamma}-\eps^{-2\gamma}\big)^{-1}\vert x\vert_2^{-4\gamma-4}\sum_{i=1}^d\vert x_i\vert^2\\
&=A_\eps^2\gamma^2\frac{\vert x\vert_2^{-4\gamma-2}}{\vert x\vert_2^{-2\gamma}-\eps^{-2\gamma}}.
\end{align*}
We may estimate the $p$-norm $|\cdot|_p$ on $\R^d$ against a constant $C_2$ times the $2$-norm $|\cdot|_2$ and get that
\begin{equation}\label{eqn-ContVarProbNotSolvable0}
\int_G\vert \nabla f_\eps(x)\vert_p^p\,\d x\leq C_2 \int_G\vert \nabla f_\eps(x)\vert_2^p\,\d x.
\end{equation}
We calculate the integral on the right as
\begin{equation}\label{eqn-ContVarProbNotSolvable1}
\begin{aligned}
\int_G\vert \nabla f_\eps(x)\vert_2^p\,\d x
&=A_\eps^{p/2}\gamma^{p/2}\int_0^\eps\Big(\frac{r^{-4\gamma-2}}{r^{-2\gamma}-\eps^{-2\gamma}}\Big)^{p/2}r^{d-1}\,\d r\\
&=A_\eps^{p/2}\gamma^{p/2}\big(\eps^{-2\gamma-2}\big)^{p/2}\eps^d\int_0^1\Big(\frac{s^{-4\gamma-2}}{s^{-2\gamma}-1}\Big)^{p/2}s^{d-1}\,\d s.
\end{aligned}
\end{equation}
The integral in the last term is obviously finite if, for some $\delta>0$, 
\begin{equation}\label{eqn-ContVarProbNotSolvable2}
\int_0^{\delta} s^{-p\gamma-p+d-1}\,\d s<\infty\qquad\text{and}\qquad\int_{1-\delta}^1\frac 1 {s^{-2\gamma}-1}\,\d s<\infty.
\end{equation}
As $p\leq 2d/(d+2)$ by assumption, it follows $(d-p)/p\geq d/2>\gamma$, which means the exponent in the first integral in \eqref{eqn-ContVarProbNotSolvable2} is greater than $-1$ and that integral is finite. For the second integral in \eqref{eqn-ContVarProbNotSolvable2}, we substitute $r=s^{-2\gamma}-1$ and estimate
\begin{align*}
\int_{1-\delta}^1\frac 1 {s^{-2\gamma}-1}\,\d s&=\frac{1}{2\gamma}\int_{1-\delta}^1r^{-1}(r+1)^{\frac{1 -2\gamma}{2\gamma}}\,\d r\leq\frac{1}{\gamma}\int_{1-\delta}^1r^{\frac{1-4\gamma}{2\gamma}}\,\d r,
\end{align*}
which is finite as $\gamma>d/4\geq 1/2$. Thus, with some constant $C_3>0$, \eqref{eqn-ContVarProbNotSolvable0} and \eqref{eqn-ContVarProbNotSolvable1} yield
$$
J^{\ssup{\rm c}}(f_\eps^2)\leq C_3 \eps^{(\gamma-d/2)p/2}\eps^{-p\gamma-p}\eps^d=C_3 \eps^{(-2p\gamma-pd-4p+4d)/4}.
$$
The assertion of Lemma \ref{lem:varprob-cont-not-solvable} follows if $-2p\gamma-pd-4p+4d>0$. This is again satisfied as $\gamma<d/2$ and $p\leq 2d/(d+2)$.
\end{proof}

Let us in the following consider the discrete variational problem. In the next statement, we write $Q_n=[-n,n]^d\cap \Z^d$ for the discrete cube of side length $2n+1$.
\begin{lemma}\label{lem:varprob-disc-not-solvable}
If $d=1$ or $\eta>d/2$, then $\chi^{\ssup{\rm d}}(Q_n)\to 0$ as $n\to\infty$. In particular, $\chi^{\ssup{\rm d}}(\Z^d)=0$.
\end{lemma}
\begin{proof}
The case $d=1$ is straightforward. We just consider the sequence of functions $f_n=n^{-1/2}\1_{[-n,n]}$. Then, up to a constant that arises from norming,
$$
\chi^{\ssup{\rm d}}(Q_n)\leq\sum_{z\in\Z}\vert f_n(z+1)-f_n(z)\vert^{\frac{2\eta}{\eta+1}}=2n^{-\frac{\eta}{\eta+1}}
$$ 
and we are done. In the case $\eta>d/2$, a more careful argument works in all dimensions. For some fixed and $L^2(G)$-normed $g\in \Ccal_c^1((-1,1)^d)$ (i.e., $g$ possesses continuous partial derivatives and has compact support), define the discretisations $$g^{\ssup n}(z)=\Big[n^{-d}\int_{[0,1]^d}g^2\Big(\frac{z+y}{n}\Big)\,\d y\Big]^{1/2},\qquad z\in\Z^d.$$ These are normed and, at least for large $n$, supported on $Q_n$. Therefore
\begin{equation}\label{eqn:disc-rescaling-vanish}
\chi^{\ssup{\rm d}}(Q_n)\leq \sum_{e\in\Ncal}\sum_{z\in\Z^d}\vert g^{\ssup n}(z+e)-g^{\ssup n}(z)\vert^{\frac{2\eta}{\eta+1}}.
\end{equation}
By H\"older's and Jensen's inequalities, we find 
\begin{equation}\label{eqn:disc-rescaling-vanish1}
\begin{aligned}
\sum_{e\in\Ncal}\sum_{z\in\Z^d} 
\vert g^{\ssup n}(z+e)-g^{\ssup n}(z)\vert^{\frac{2\eta}{\eta+1}}
&\leq n^{-\frac{d\eta}{\eta+1}}\sum_{\substack{z\in\Z^d\\e\in\Ncal}} 
\Big[\int_{[0,1]^d} \Big\vert g\Big(\frac{z+y+e}{n}\Big)-g\Big(\frac{z+y}{n}\Big)\Big\vert^2\,\d y\Big]^{\frac{\eta}{\eta+1}}\\
&\lesssim \Big[\sum_ {e\in \Ncal}\sum_{z\in\Z^d}\int_{[0,1]^d} \Big\vert g\Big(\frac{z+y+e}{n}\Big)-g\Big(\frac{z+y}{n}\Big)\Big\vert^2\,\d y\Big]^{\frac{\eta}{\eta+1}}\\
&=n^{\frac{d}{\eta+1}}\Big[\sum_{e\in\Ncal}\int_{\R^d}\Big\vert g\Big(y+\frac{e}{n}\Big)-g(y)\Big\vert^2\,\d y\Big]^{\frac{\eta}{\eta+1}}.
\end{aligned}
\end{equation}
Replacing the difference under the last integral according to the fundamental theorem of calculus, we see that
$$
\mbox{r.h.s.~of \eqref{eqn:disc-rescaling-vanish1}}=n^{\frac{d-2\eta}{\eta+1}}\Big[\sum_{i=1}^d\int_{\R^d}\int_0^1\Big\vert\partial_ig\Big(y+\frac{se_i}{n}\Big)\Big\vert^2\,\d s\,\d y\Big]^{\frac{\eta}{\eta+1}}
=n^{\frac{d-2\eta}{\eta+1}}\Big[\sum_{i=1}^d\int_{\R^d}\vert\partial_ig(y)\vert^2\,\d y\Big]^{\frac{\eta}{\eta+1}},
$$
where the term in parentheses is obviously finite. This shows that the right-hand side in \eqref{eqn:disc-rescaling-vanish} tends to $0$ as $n\to\infty$ and thus completes the proof of Lemma~\ref{lem:varprob-disc-not-solvable}.
\end{proof}

\begin{lemma}\label{lem:varprob-disc-solvable}
If $d>1$ and $\eta\leq d/2$, then $\chi^{\ssup{\rm d}}(\Z^d)>0$.
\end{lemma}

\begin{proof}
As $\chi^{\ssup{\rm d}}(\Z^d)$ is non-increasing with $\eta$, it suffices to consider the case $\eta=d/2$ and we abbreviate $p=\smfrac{2\eta}{\eta+1}=\smfrac{2d}{d+2}$. We prove the case $d=2$ and $d\geq 3$ separately. 

The proofs rely on a discrete Sobolev inequality the reader may find in \cite[Lemma 3.2.10]{S10}, see also \cite{KS12}. It states that in dimension $d\geq 2$, we have for all $g\colon\Z^d\to[0,\infty)$ with $g(z)\to 0$ as $\vert z\vert\to\infty$
\begin{equation}\label{eqn-discreteSobolev}
\sum_{z\in\Z^d}g(z)^{\frac d{d-1}}\leq\Big(\sum_{z\in\Z^d,e\in\Ncal}\vert g(z+e)-g(z)\vert\Big)^{\frac d{d-1}}.
\end{equation}

\emph{Case} $d=2$:
Here, $p=1$ and $\frac{d}{d-1}=2$. It follows directly from \eqref{eqn-discreteSobolev} that $\sum_{z,e}\vert f(z+e)-f(z)\vert\geq 1$ for all normed functions $f\in \ell^2(\Z^2)$. This shows the assertion.

\emph{Case} $d\geq 3$:
Take an arbitrary normed function $f\in \ell^2(\Z^d)$. Without loss of generality, we may assume that $f$ is non-negative. Put $\alpha=\frac{2d-2}d>1$, consider \eqref{eqn-discreteSobolev} with $g=f^\alpha$ and apply the mean value theorem to each summand. It follows

\begin{equation*}
1\leq\sum_{z\in\Z^d,e\in\Ncal}\vert f(z+e)^\alpha-f(z)^\alpha\vert\leq \sum_{z\in\Z^d,e\in\Ncal}\alpha\vert f(z+e)-f(z)\vert(f(z+e)^{\alpha-1}+f(z)^{\alpha-1}),
\end{equation*}
which in combination with H\"older's inequality yields
\begin{equation*}
1\leq 2d\alpha\Big(\sum_{z\in\Z^d,e\in\Ncal}\vert f(z+e)-f(z)\vert^p\Big)^{\frac 1p}\Big(\sum_{z\in\Z^d}f(z)^{\frac{(\alpha-1)p}{p-1}}\Big)^{\frac{p-1}p}.
\end{equation*}
The second sum is equal to $1$ as $\frac{(\alpha-1)p}{p-1}=2$ due to the choices $p=\smfrac{2d}{d+2}$ and $\alpha=\frac{2d-2}d$. Rearrangement of the equation above yields the desired result.
\end{proof}

\begin{lemma}\label{lem:disc-varprob-converge}
Assume $\eta\geq\frac d2$. Consider, for $n\in\N$, the boxes $Q_n=[-n,n]^d\cap\Z^d$. Then
\begin{equation}
\lim_{n\to\infty}\chi^{\ssup{\rm d}}(Q_n)=\chi^{\ssup{\rm d}}(\Z^d).
\end{equation}
\end{lemma}

\begin{proof}
As obviously $\chi^{\ssup{\rm d}}(Q_n)\geq \chi^{\ssup{\rm d}}(\Z^d)$, it remains to show that
\begin{equation}\label{eqn:disc-varprob-converge-limsup}
\limsup_{n\to\infty}\chi^{\ssup{\rm d}}(Q_n)\leq\chi^{\ssup{\rm d}}(\Z^d).
\end{equation}
To that end, write $p=\frac{2\eta}{\eta+1}$ and choose some arbitrarily small $\delta>0$. Then there exists some normed $g\in\ell^2(\Z^d)$ such that
\begin{equation}
\sum_{e\in\Ncal}\sum_{z\in\Z^d}\vert g(z+e)-g(z)\vert^p\leq \chi^{\ssup{\rm d}}(\Z^d)+\delta.
\end{equation}
We will now cut off this $g$ in a sufficiently smooth way to obtain an upper bound for $\chi^{\ssup{\rm d}}(Q_n)$. Define $\xi\colon\R^d\to\R$ by
\begin{equation}
\xi(x)=\begin{cases}
1,\qquad&\vert x\vert_2\leq1,\\
2-\vert x\vert_2,\qquad &1<\vert x\vert_2<2,\\
0,\qquad&\vert x\vert_2\geq2.
\end{cases}
\end{equation}
Then, the norm $r(n)$ of $g_n$ defined as $g_n(z)=g(z)\xi(z/n)$, $z\in\Z^d$, obviously tends to $1$ as $n\to\infty$. Moreover, we have in the case $p\leq 1$,
\begin{align*}
\chi^{\ssup{\rm d}}(Q_{2n+1})&\leq r(n)^{-p}\sum_{e\in\Ncal}\sum_{z\in\Z^d}\vert g_n(z+e)-g_n(z)\vert^p\\
&\leq r(n)^{-p}\sum_{e\in\Ncal}\sum_{z\in\Z^d}\vert g(z+e)-g(z)\vert^p\xi\big((z+e)/n\big)^p\\
&\qquad + r(n)^{-p}\sum_{e\in\Ncal}\sum_{z\in\Z^d}\vert g(z)\vert^p\big\vert\xi\big((z+e)/n\big)-\xi(z/n)\big\vert^p.
\end{align*}
In the case $p>1$, we obtain as an analogous estimate by Minkowski's inequality
\begin{align*}
\big(\chi^{\ssup{\rm d}}(Q_{2n+1})\big)^{1/p}&\leq \frac1{r(n)}\Big(\sum_{e\in\Ncal}\sum_{z\in\Z^d}\vert g_n(z+e)-g_n(z)\vert^p\Big)^{1/p}\\
&\leq \frac1{r(n)}\Big(\sum_{e\in\Ncal}\sum_{z\in\Z^d}\vert g(z+e)-g(z)\vert^p\xi\big((z+e)/n\big)^p\Big)^{1/p}\\
&\qquad+ \frac1{r(n)}\Big(\sum_{e\in\Ncal}\sum_{z\in\Z^d}\vert g(z)\vert^p\big\vert\xi\big((z+e)/n\big)-\xi(z/n)\big\vert^p\Big)^{1/p}.
\end{align*}
In both cases, the first term on the right-hand side clearly tends to $\chi^{\ssup{\rm d}}(\Z^d)+\delta$ and its $1/p$-th power, respectively. As $\delta$ was chosen arbitrarily small, it is enough to show that the sums in the respective second terms vanish as $n\to\infty$. For some positive constants $c_1<c_2$ that depend on dimension only, it is obvious that 
\begin{equation}
\vert\xi\big((z+e)/n\big)-\xi(z/n)\vert=0\qquad\text{if } z\notin Q_{\lfloor c_2n\rfloor}\setminus Q_{\lfloor c_1n\rfloor}, e\in\Ncal,
\end{equation}
if $n$ is large enough. Moreover, the same difference is of course always bounded by $n^{-1}$. Therefore, we may estimate, with the help of H\"older's inequality,
$$
\begin{aligned}
\sum_{e\in\Ncal}\sum_{z\in\Z^d}\vert g(z)\vert^p\big\vert\xi\big((z+e)/n\big)-\xi(z/n)\big\vert^p
&\leq  \Big(d\sum_{z\in\Z^d\setminus Q_{\lfloor c_1n\rfloor}}\vert g(z)\vert^2\Big)^{p/2}\Big(d\sum_{z\in Q_{\lfloor c_2n\rfloor}}n^{-\frac{2p}{2-p}}\Big)^{\frac{2-p}2}\\
&\leq c_3\Big(\sum_{z\in\Z^d\setminus Q_{\lfloor c_1n\rfloor}}\vert g(z)\vert^2\Big)^{p/2}\Big(n^{d-\frac{2p}{2-p}}\Big)^{\frac{2-p}2}
\end{aligned}
$$
with a constant $c_3>0$ that also depends on the dimension only. As $g$ was assumed to be $\ell^2$-normed, the assertion follows if only $d-\frac{2p}{2-p}\leq 0$. But this is tantamount to $\eta\geq\frac d2$.
\end{proof}

\section{Auxiliary large deviation statements}\label{sec:ldp-growing-prelim}

In this section, we prove two tools that will be important for the proof of the main results later and have also some interest in their own right: a rescaled LDP of Donsker-Varadhan-G\"artner type with deterministic conductances in Section~\ref{subsec:ldp-rescaled-fix}, and a version of an LDP for the conductances in Section~\ref{subsec:asycond}.

\subsection{Donsker-Varadhan-G\"artner type LDP for deterministically rescaling conductances.}\label{subsec:ldp-rescaled-fix}
In this section, we prove an LDP for the rescaled local times, $L_t$, for a time-dependent sequence of conductances that rescale to some fixed profile. More precisely, for  $\varphi\colon G\times\Ncal\to(0,\infty)$ we define its \lq unscaled\rq\ version by
\begin{equation}\label{varphitdef}
\varphi_t(z,e)=\int_{[0,1]^d}\varphi\Big(\frac{z+y}{\alpha_t},e\Big)\,\d y,\qquad z\in B_t, e\in\Ncal.
\end{equation}
Here, we recall that $B_t=\alpha_tG\cap \Z^d$. The following is an extension of \cite[Lemma 3.1]{GKS07} from $\varphi\equiv 1$ (i.e., simple random walk) to a much larger class of conductances.

\begin{prop}\label{lem:rescaled-fix-ldp} Fix $\varphi\colon G\times\Ncal\to(0,\infty)$ such that $\varphi(\cdot,e)\in \Ccal(G)$ for any $e\in\Ncal$ and such that $m\leq \varphi\leq M$ for some $0<m<M<\infty$. Then the rescaled local times $L_t$ under $\P^{\beta_t^{-1}\varphi_t}$ conditioned on the event $\{\supp(\ell_t)\subset \alpha_t G\}$ satisfy an LDP on $\Fcal$ with scale $t\alpha_t^{-2}\beta_t^{-1}$ and rate function $I^{\ssup{\rm c}}_{\varphi,0}=I^{\ssup{\rm c}}_\varphi-\inf_\Fcal I^{\ssup{\rm c}}_\varphi$ where
\begin{equation}\label{DVratefct}
I^{\ssup{\rm c}}_\varphi(f^2)=\begin{cases}
\sum_{e\in\Ncal}\int_G\varphi(y,e)\big(\partial_e f\big)^2(y)\,\d y,\qquad&\mbox{if }f\in H^1_0(G)\\
\infty,\qquad&\text{else.}
               \end{cases}
\end{equation}
Here, the space $\Fcal$ is equipped with the weak topology of test integrals against bounded continuous functions $V\colon G\to\R$.
\end{prop}

We follow partly the proof of \cite[Lemma 3.1]{GKS07} and use the G\"artner-Ellis theorem, i.e., we identify the exponential rate of exponential test integrals against bounded continuous functions. However, we cannot rely on the local central limit theorem here, but rather use an eigenvalue expansion. Hence we will have to control the principal eigenvalue and the corresponding eigenfunction. This will be done in Lemmas~\ref{lem:eigenvalues-converge} and \ref{lem-eigenfunction-decay}, respectively, which are the two main steps in the proof of Proposition~\ref{lem:rescaled-fix-ldp}.

For $V$ in $\Ccal_{\rm b}(G)$, the set of bounded continuous functions $G\to\R$, we define its unscaled discretisation analogously to \eqref{varphitdef}:
\begin{equation}\label{eqn:vtdef}
V_t(z)=\int_{[0,1]^d}\d y\,V\Big(\frac{z+y}{\alpha_t}\Big),\quad z\in\alpha_tG\cap\Z^d
\end{equation}
Then, denote by $\lambda^{\ssup t}(\varphi,V)$ the principal (i.e., smallest) eigenvalue of $-\alpha_t^2\Delta^{\varphi_t}+V_t$ in $B_t$ with zero boundary condition. Analogously, we call $\lambda_1(\varphi,V)$ the largest eigenvalue of the continuous operator
$$
-\Delta^\varphi+V=-\nabla^\ast A\nabla+V
$$
on $H_0^1(G)$, where the space-dependent matrix $A$ is given by 
$$
A_{ij}(y)=\delta_{ij}\varphi(y,e_i),\qquad y\in G,\,i,j\in\{1,\ldots,d\}.
$$ 
The Rayleigh-Ritz principle can be written as
$$
\lambda_1(\varphi,V)=\inf_{f\in \Fcal}\{I^{\ssup{\rm c}}_\varphi(f^2)+(Vf,f)\}.
$$ 
It turns out that the discrete eigenvalue converges towards the continuous one if the discrete region grows.

\begin{lemma}\label{lem:eigenvalues-converge}Fix $\varphi$ as in Proposition~\ref{lem:rescaled-fix-ldp}. Then, for any $V\in\Ccal_{\rm b}(G)$, 
$$
\lim_{t\to\infty}\lambda_1^{\ssup t}(\varphi,V)=\lambda_1(\varphi,V).
$$
\end{lemma}

\begin{proof}Let us write $\lambda_1^{\ssup t}$ and $\lambda_1$ instead of $\lambda_1^{\ssup t}(\varphi,V)$ and $\lambda_1(\varphi,V)$.
We need to show that
\begin{eqnarray}
\limsup_{t\to\infty}\lambda_1^{\ssup t}&\leq& I^{\ssup{\rm c}}_\varphi(f^2)+(Vf,f),\qquad\text{for all }f\in\Fcal\text{ and}\label{eqn-lambdaConv1}\\
\liminf_{t\to\infty}\lambda_1^{\ssup t}&\geq& \lambda_1.\label{eqn-lambdaConv2}
\end{eqnarray}
\emph{Proof of \eqref{eqn-lambdaConv1}.}  This equation is only non-trivial for functions $f$ in $H^1_0(G)$, so let $f$ be such a function. As $\Ccal_{\rm c}^\infty(G)$ is dense in $H_0^1(G)$, there exists a sequence of functions $f^{\ssup n}\in \Ccal_{\rm c}^\infty(G)$ with $\Vert f-f^{\ssup n}\Vert_{H^1}\leq\frac 1 n$ for any $n\in\N$. Moreover, we may assume that the $H^1$-norms of all these functions $f^{\ssup n}$ are bounded by some constant $N>0$. With the convention
$$
f_t(z)^2=\alpha_t^{-d}\int_{[0,1)^d}f\Big(\frac{z+y}{\alpha_t}\Big)^2\,\d y,\quad z\in\alpha_tG\cap\Z^d,
$$
we have by the Rayleigh-Ritz formula
\begin{equation}\label{eqn-lambdaConv3}
\lambda_1^{\ssup t}\leq\alpha_t^2\sum_{z\in\alpha_tG\cap\Z^d,e\in \Ncal}\varphi_t(z,e)(f_t^{\ssup n}(z+e)-f_t^{\ssup n}(z))^2+\sum_{z\in\alpha_tG\cap\Z^d} V_t(z)f_t^{\ssup n}(z)^2.
\end{equation}
We estimate the first sum by
\begin{align*}
&\sum_{z,e}\varphi_t(z,e)(f_t^{\ssup n}(z+e)-f_t^{\ssup n}(z))^2\\
&=\alpha_t^{-d}\sum_{z,e}\varphi_t(z,e)\Big[\Big(\int_{[0,1)^d}f^{\ssup n}\Big(\frac {z+x+e}{\alpha_t}\Big)^2\,\d x\Big)^\frac 1 2-\Big(\int_{[0,1)^d}f^{\ssup n}\Big(\frac{z+x}{\alpha_t}\Big)^2\,\d x\Big)^\frac 1 2\Big]^2\\
&\leq\alpha_t^{-d}\sum_{z,e}\varphi_t(z,e) \int_{[0,1)^d} \Big[f^{\ssup n}\Big(\frac {z+x+e}{\alpha_t}\Big)-f^{\ssup n}\Big(\frac{z+x}{\alpha_t}\Big)\Big]^2\,\d x\\
&\leq\sum_{e}\int_G\varphi_t(\lfloor\alpha_ty\rfloor,e)\Big[f^{\ssup n}\Big(y+\frac{e}{\alpha_t}\Big)-f^{\ssup n}\Big(y\Big)\Big]^2\,\d y\\
&\leq\alpha_t^{-2}\sum_{e}\int_0^1\int_G\varphi_t(\lfloor\alpha_ty\rfloor,e)\partial_ef^{\ssup n}\Big(y+\frac{se}{\alpha_t}\Big)^2\,\d y\,\d s,
\end{align*}
making use of H\"older's inequality in the second step, an integral substitution in the third, and the fundamental theorem of calculus combined with Jensen's inequality and Fubini's theorem in the fourth. From here, we may estimate by the triangle inequality
\begin{align*}
&\alpha_t^2\sum_{z,e}\varphi_t(z,e)(f_t^{\ssup n}(z+e)-f_t^{\ssup n}(z))^2\leq\sum_{e}\int_G\varphi(y,e)(\partial_ef)^2(y)\,\d y+R_1+R_2+R_3
\end{align*}
with
\begin{align*}
R_1&=\sum_{e}\int_0^1\int_G\varphi_t(\lfloor\alpha_ty\rfloor,e)\Big[\partial_ef^{\ssup n}\Big(y+\frac{se}{\alpha_t}\Big)^2-\partial_ef^{\ssup n}(y)^2\Big]\,\d y\,\d s,\\
R_2&=\sum_{e}\int_G\varphi_t(\lfloor\alpha_ty\rfloor,e)\Big[\partial_e f^{\ssup n}(y)^2-\partial_ef(y)^2\Big]\,\d y,\\
R_3&=\sum_{e}\int_G\Big[\varphi_t(\lfloor\alpha_ty\rfloor,e)-\varphi(y,e)\Big](\partial_ef)^2(y)\,\d y.
\end{align*}
Firstly, by H\"older's and Minkowski's inequalities, we have 
$$
\vert R_1\vert\leq 2MN \sum_{e}\int_0^1\int_G \Big[\partial_ef^{\ssup n}\Big(y+\frac{se}{\alpha_t}\Big)-\partial_ef^{\ssup n}(y)\Big]^2\,\d y\,\d s,
$$ 
which converges to zero with $t\to\infty$ as $f^{\ssup n}$ is bounded and continuous. Again with H\"older's and Minkowski's inequalities, we find that $\vert R_2\vert\leq\frac{2MN}n$. The term $R_3$ goes to zero with $t\to\infty$ as $\varphi$ is bounded and continuous. Finally, convergence of $\sum_{z\in\alpha_tG\cap\Z^d} V_t(z)f_t^{\ssup n}(z)^2$ towards $(Vf,f)$ follows in a similar way by dint of Lebesgue's theorem. This means we have
\begin{equation}\label{eqn-lambdaConv4}
\limsup_{t\to\infty}\lambda_1^{\ssup t}\leq I^{\ssup{\rm c}}_\varphi(f^2)+(Vf,f)+\frac{2MN}n
\end{equation}
for all $n\in\N$ and $f\in H_0^1(G)$. Letting $n\to\infty$, we obtain \eqref{eqn-lambdaConv1}.

\emph{Proof of \eqref{eqn-lambdaConv2}.} We denote by $v_t$ the $\ell^2$-normed and positive principal eigenfunction of the operator $-\alpha_t^2\Delta^{\varphi_t}+V_t$ in $B_t$ with zero boundary condition corresponding to the eigenvalue $\lambda_1^{\ssup t}$. The strategy is to construct a sequence of functions $f_t\in H_0^1(G)$ satisfying
\begin{eqnarray}
-\alpha_t^2\big(\Delta^{\varphi_t}v_t,v_t\big)&=&I^{\ssup{\rm c}}_\varphi\big(f_t^2\big),\label{eqn-SameEnergy}\\
\liminf_{t\to\infty}\big(V_tv_t,v_t\big)&=&\liminf_{t\to\infty}\big(Vf_t,f_t\big),\label{eqn-SamePotential}\\
\lim_{t\to\infty}\Vert f_t\Vert_2&=&1.\label{eqn-NormsConverge}
\end{eqnarray}
Given such a sequence, we then easily deduce
\begin{align*}
\liminf_{t\to\infty}\Big(-\alpha_t^2\big(\Delta^{\varphi_t}v_t,v_t\big)+\big(V_tv_t,v_t\big)\Big)
&=\liminf_{t\to\infty}\Big(\Vert f_t\Vert_2^{-2}I^{\ssup{\rm c}}_\varphi\big(f_t^2\big)+\Vert f_t\Vert_2^{-2}\big(Vf_t,f_t\big)\Big)\\
&\geq\inf_{f\in \Fcal}\Big[I^{\ssup{\rm c}}_\varphi(f^2)+\big(Vf,f\big)\Big]=\lambda_1,
\end{align*}
which implies \eqref{eqn-lambdaConv2} as the $v_t$ are the discrete principal eigenfunctions.

The construction uses a finite element approach which was used in a similar way in \cite{BK11} and involves an extension of the discrete eigenfunctions $v_t$ onto the continuous space $\alpha_tG$ by linear interpolation along certain simplices and subsequent rescaling. The unit cube $K=[0,1]^d$ is split into $d!$ simplices as follows: For each permutation $\sigma\in\Sigma_d$ of the set $\{1,\ldots, d\}$, let $T_\sigma$ denote the interior of the convex hull of the integer vertices $0,e_{\sigma(1)},e_{\sigma(1)}+e_{\sigma(2)},\ldots,e_{\sigma(1)}+\ldots+e_{\sigma(d)}$ with $e_i$ the $i$-th unit vector. Consequently, the sets $T_\sigma$ with $\sigma\in\Sigma_d$ are pairwise disjoint. For Lebesgue-almost all $x\in\R$  we find $\sigma_x\in\Sigma_d$ such that $x-\lfloor x\rfloor$ is in $T_{\sigma_x}$. We may consequently define, for $t>0$, almost all $x\in\alpha_tG$ and $i\in\{1,\ldots,d\}$,
\begin{equation*}
g^{\ssup t}_i(x)=\big(x_{\sigma_x(i)}-\lfloor x_{\sigma_x(i)}\rfloor\big)\Big[v_t\big(\lfloor x\rfloor+e_{\sigma_x(1)}+\ldots+e_{\sigma_x(i)}\big)-v_t\big(\lfloor x\rfloor+e_{\sigma_x(1)}+\ldots+e_{\sigma_x(i-1)}\big)\Big].
\end{equation*}
Let us now define the sequence $f_t$ with the desired properties. If $y\in G$ with $\alpha_ty-\lfloor\alpha_ty\rfloor$ belonging to some $T_\sigma$, let
\begin{equation}\label{eqn-InterpolDef}
f_t(y)=\alpha_t^{d/2}v_t\big(\lfloor \alpha_ty\rfloor\big)+\alpha_t^{d/2}\sum_{i=1}^dg^{\ssup t}_i(\alpha_ty).
\end{equation}
We may extend the functions $f_t$ continuously to the whole space $G$ as is shown in \cite{BK11}, and they are clearly differentiable in all points $y\in G$ with $\alpha_ty-\lfloor\alpha_ty\rfloor$ belonging to some $T_\sigma$, which means $f_t\in H^1_0(G)$. It is easily seen that the functions $f_t$ satisfy \eqref{eqn-SameEnergy}: For almost all $y\in G$,
\begin{equation}
\partial_ef_t(y)=\alpha_t^{1+d/2}\big[v_t(\lfloor\alpha_ty\rfloor+e)-v_t(\lfloor\alpha_ty\rfloor)\big],\qquad e\in\Ncal, t>0.
\end{equation}
In particular, $\partial_e f_t$ is almost everywhere constant on the boxes $\alpha_t^{-1}(z+[0,1]^d)$ with $z\in\alpha_tG\cap\Z^d$, thus
\begin{align*}
\alpha_t^2\big(\Delta^{\varphi_t}v_t,v_t\big)&=\alpha_t^{d+2}\sum_{e\in\Ncal}\int_G\varphi_t(\lfloor\alpha_ty\rfloor,e)\big[v_t(\lfloor\alpha_ty\rfloor+e)-v_t(\lfloor\alpha_ty\rfloor)\big]^2\,\d y\\
&=\sum_{e\in\Ncal}\int_G\varphi_t(\lfloor\alpha_ty\rfloor,e)\big(\partial_ef_t(y))^2\,\d y\\
&=\sum_{e\in\Ncal}\int_G\varphi_t(\lfloor\alpha_ty\rfloor,e)\big(\partial_ef_t(y))^2\,\d y=-I^{\ssup{\rm c}}_\varphi\big(f_t^2\big).
\end{align*}
Let us in a next step show that the functions $f_t$ also satisfy \eqref{eqn-NormsConverge}. By the triangle inequality applied to \eqref{eqn-InterpolDef}, it is enough to show that the $L^2(G)$-norms of each sequence of functions $\alpha_t^{d/2}g^{\ssup t}_i(\alpha_t\cdot)$, $i=1,\ldots,d$,  vanish as $t\to\infty$. We calculate 
\begin{align}\label{eqn:interpolation-residual}
&\alpha_t^d\Vert \sum_{i=1}^dg^{\ssup t}_i(\alpha_t\cdot)\Vert_2^2=\alpha_t^d\int_G(\sum_{i=1}^dg^{\ssup t}_i(\alpha_ty))^2\,\d y\notag\\
&\leq\sum_{i=1}^d\Big(\int_{\alpha_tG}\big(y_{\sigma_y(i)}-\lfloor y_{\sigma_y(i)}\rfloor\big)^2\Big[v_t\big(\lfloor y\rfloor+e_{\sigma_y(1)}+\ldots+e_{\sigma_y(i)}\big)\notag\\
&\qquad -v_t\big(\lfloor y\rfloor+e_{\sigma_y(1)}+\ldots+e_{\sigma_y(i-1)}\big)\Big]^2\,\d y\Big)\notag\\
&\leq\sum_{i=1}^d\Big(\int_{\alpha_tG}\Big[v_t\big(\lfloor y\rfloor+e_{\sigma_y(1)}+\ldots+e_{\sigma_y(i)}\big)-v_t\big(\lfloor y\rfloor+e_{\sigma_y(1)}+\ldots+e_{\sigma_y(i-1)}\big)\Big]^2\,\d y\Big)\notag\\
&=\sum_{e\in\Ncal}\int_{\alpha_tG}\Big[v_t\big(\lfloor y\rfloor+e\big)-v_t\big(\lfloor y\rfloor\big)\Big]^2\,\d y\notag\\
&\leq m^{-1}\sum_{z\in\alpha_tG\cap\Z^d,e\in\Ncal}\varphi_t(z,e)\big[v_t(z+e)-v_t(z)\big]^2.
\end{align}
The last expression must converge to zero as $t\to\infty$ as the converse would imply $$\limsup_{t\to\infty}\alpha_t^2\big(-\Delta^{\varphi_t}v_t,v_t\big)=\infty$$ in contradiction to \eqref{eqn-lambdaConv1} that we have already proven. Equation \eqref{eqn-SamePotential} is seen as follows. By the triangle inequality,
\begin{align*}
\Big\vert\big(V_tv_t,v_t\big)-\big(Vf_t,f_t\big)\Big\vert&\leq\sum_{z\in\alpha_tG\cap\Z^d}\big\vert V_t(z)-V(z/\alpha_t)\big\vert\big(v_t(z)\big)^2\\
&\qquad +\int_G\big\vert V(y)\big\vert\Big\vert\alpha_t^d\big(v_t(\lfloor\alpha_ty\rfloor)\big)^2-\big(f_t(y)\big)^2\Big\vert\,\d y,
\end{align*}
where the second term vanishes with $t\to\infty$ due to \eqref{eqn-NormsConverge} and the fact that $V$ is bounded. As $v_t$ is normed, we obtain an upper bound for the first term by replacing $\big(v_t(z)\big)^2$ with $\delta_z(z_t)$ where $$z_t=\arg\max \big\vert V_t(z)-V(z/\alpha_t)\big\vert.$$ Then, \eqref{eqn-SamePotential} follows considering that $\big\vert V_t(z_t)-V(z_t/\alpha_t)\big\vert\to 0$ as $V$ is uniformly continuous. This finishes the proof of \eqref{eqn-lambdaConv2}.
\end{proof}

Recall that $v_t$ denotes the $\ell^2$-normed and positive principal eigenfunction of $-\alpha_t^2\Delta^{\varphi_t}+V_t$ in $B_t=\alpha_t G\cap\Z^d$ with zero boundary condition corresponding to the eigenvalue $\lambda_1^{\ssup t}=\lambda_1^{\ssup t}(\varphi,V)$.

\begin{lemma}\label{lem-eigenfunction-decay}
Under the assumptions of Lemma~\ref{lem:eigenvalues-converge},
$$
\liminf_{t\to\infty}\frac{\beta_t\alpha_t^2}{t}\log v_t(0)\geq 0.
$$
\end{lemma}
\begin{proof}
We treat the cases $d=1$ and $d\geq 2$ separately. 

\emph{Case $d=1$}: There is a unique $L^2$-normed $g\in H_0^1(G)$ such that 
$$
I^{\ssup{\rm c}}_\varphi(g^2)+(Vg,g)=\lambda_1(\varphi,V)
$$ 
and $g$ is strictly positive in the sense that for any compact set $K\subset G$ there exists $\delta>0$ such that $g>\delta$ almost everywhere on $K$, thus $g>\delta_1$ on $[-\delta_2,\delta_2]^d\subset G$ for some fixed positive constants $\delta_1,\delta_2$. This follows from the spectral theorem for uniformly elliptic operators (compare e.g.~\cite{Z90}), note that $\varphi$ is continuous and $0<m\leq\varphi\leq M<\infty$ by assumption. Let $f_t$ be the interpolating sequence we have constructed in the proof of the previous lemma. We now  show that $f_t$ converges to $g$ in $L^\infty$ towards $g$ as $t\to\infty$. As every sequence $(f_{t_k})_{k\in\N}$ is a minimizing sequence with respect to the Dirichlet energy associated with $-\alpha_t\Delta^\varphi_t+V_t$, and $\varphi$ is bounded away from zero, this sequence is bounded in $H_0^1(G)$ and therefore 
admits a weakly convergent subsequence that we also denote by $(f_{t_k})_{k\in\N}$. By the Rellich-Kondrashov theorem (Theorem \ref{thm:rellich-kondrashov}) in the special case $p=2$, $d=1$, we have $f_{t_k}\to f$ in $L^\infty$ for some $f\in L^2(G)$. As the minimiser $g$ is unique and 
$$
I^{\ssup{\rm c}}_\varphi(f^2)+(Vf,f)\leq \liminf_{k\to\infty} I^{\ssup{\rm c}}_\varphi\big(f_{t_k}^2\big)+(Vf_{t_k},f_{t_k})
$$ 
by lower semicontinuity of $I^{\ssup{\rm c}}_\varphi$ and continuity of $V$, we have $f=g$. For $t$ large enough, we have consequently $f_t>\delta_1/2$ on $[-\delta_2,\delta_2]^d$. As $f_t$ interpolates $\alpha_t^{d/2}v_t$, this also implies 
that $\alpha_t^{d/2}v_t(0)>\delta_1/2$. The decay of $v_t(0)$ is therefore only polynomial in $t$ and the assertion is shown.

\emph{Case $d\geq 2$}: As $v_t$ is an eigenfunction of $-\alpha_t^2\Delta^{\varphi_t}+V_t$ corresponding to the eigenvalue $\lambda^{\ssup t}(\varphi,V)$, we have
$$
v_t(0)=\e^{-\lambda_t(V)}\big(\exp\{\alpha_t^2\Delta^{\varphi_t}-V_t\}v_t\big)(0)=\e^{-\lambda_t(V)}\E^{\alpha_t^2\varphi_t}_0\Big[\exp\Big\{-\int_0^1V_t(X_s)\,\d s\Big\}v_t(X_1)\Big].
$$
Abbreviating $v_t^\ast=\max_{\alpha_tG\cap\Z^d} v_t$ and $V_\ast=\sup_GV$, we estimate
\begin{equation*}
v_t(0)\geq v_t^\ast\e^{-\lambda_t(V)-V_\ast}\min_{x\in\alpha_tG\cap\Z^d}\P^{\alpha_t^2\varphi_t}_0\big(X_1=x\big).
\end{equation*}
As $v_t$ is normed, the decay of its maximal value is slower than exponential as $t\to\infty$, so we only need to consider the exponential decay of the probability term above. With $\vert \cdot\vert=\vert \cdot\vert_1$ denoting the lattice distance, $r$ the radius of the smallest ball to contain $G$ and $S_1$ the random number of jumps a random walk makes up to time $1$, we have
$$
\P^{\alpha_t^2\varphi_t}_0\big(X_1=x\big)=\sum_{k=\vert x\vert}^\infty\P^{\alpha_t^2\varphi_t}_0\big(X_1=x,S_1=k\big)\geq \Big(\frac{2dM}{m}\Big)^{-2dr\lceil\alpha_t\rceil}\P^{\alpha_t^2\varphi_t}_0\big(S_1\geq\vert x\vert\big),
$$
as jump times are independent from jump directions and the random walk can always reach the vertex $x$ by making its last $2dr\lceil\alpha_t\rceil$ steps in the \lq right\rq~direction, since this is the maximum lattice distance within $\alpha_tG$. Certainly the probability of the random walk generated by $\Delta^{\alpha_t^2\varphi}$ to make at least $\vert x\vert$ jumps dominates the probability of the slower simple random walk generated by $\alpha_t^2m\Delta$ to make at least $2dr\alpha_t$ jumps. Thus,
\begin{align*}
\min_{x\in\alpha_tG\cap\Z^d}\P^{\alpha_t^2\varphi_t}_0\big(X_1=x\big)&\geq \Big(\frac{2dM}{m}\Big)^{-2dr\lceil\alpha_t\rceil}\P^{\alpha_t^2m}_0\big(S_1\geq 2dr\alpha_t)\\
&= \Big(\frac{2dM}{m}\Big)^{-2dr\lceil\alpha_t\rceil}\e^{-2d\alpha_t^2m}\sum_{k=2dr\lceil\alpha_t\rceil}^\infty \frac{(2d\alpha_t^2m)^k}{k!}\\
&\geq \Big(\frac{2dM}{m}\Big)^{-2dr\lceil\alpha_t\rceil}\e^{-2d\alpha_t^2m}\frac{(2d\alpha_t^2m)^{2dr\lceil\alpha_t\rceil}}{(2dr\lceil\alpha_t\rceil)!}.
\end{align*}
In the last line, we observe that the fraction in the end is greater than $1$ if $\alpha_t$ is large enough. Therefore,
\begin{equation}\label{eqn-decay-eigenfunction}
\log v_t(0)\geq -2d\alpha_t^2m+o(\alpha_t^2).
\end{equation}
As we are in the case $d\geq 2$ and we have chosen $\beta_t\gg 1$ such that $\alpha_t^d\beta_t^\eta=t\alpha_t^{-2}\beta_t^{-1}$, we may conclude $\alpha_t^2\ll t\alpha_t^{-2}\beta_t^{-1}$. The assertion follows dividing \eqref{eqn-decay-eigenfunction} by $t\alpha_t^{-2}\beta_t^{-1}$ and passing to the limit inferior. 
\end{proof}

\begin{proof}[Proof of Proposition \ref{lem:rescaled-fix-ldp}] The proof of the LDP in Proposition~\ref{lem:rescaled-fix-ldp} relies on the G\"artner-Ellis-theorem (e.g., in \cite{DZ98}). It will be sufficient to show that
\begin{equation}\label{eqn:rescaled-fix-cumulant-generating}
\lim_{t\to\infty}\frac{\beta_t\alpha_t^2}{t}\log\E_z^{\beta_t^{-1}\varphi_t}\Big[\exp\Big\{-\frac{t}{\beta_t\alpha_t^2}\int_G V(y)L_t(y)\,\d y\Big\}\,\Big\vert\, X_{[0,t]}\subset\alpha_tG\Big]=-\lambda_1(\varphi,V)+\lambda_1(\varphi,0).
\end{equation}
for all $V\in \Ccal_{\rm b}(G)$. Then, by the G\"artner-Ellis theorem, the desired result follows as the Legendre transform of the rate function $I_{\varphi,0}^{\ssup{\rm c}}$ is given by
\begin{equation}
V\mapsto \sup_{g^2\in\Fcal}\big\{(V,g^2)-I_{\varphi,0}^{\ssup{\rm c}}(g^2)\big\}=\sup_{g^2\in\Fcal}\big\{(V,g^2)-I_{\varphi}^{\ssup{\rm c}}(g^2)\big\}+\lambda_1(\varphi,0)=-\lambda_1(\varphi,V)+\lambda_1(\varphi,0).
\end{equation}
Here, $(\cdot,\cdot)$ denotes the $L^2(G)$-scalar product and we have made use of the well-established fact that the eigenvalue $\lambda_1(\varphi,V)$ satisfies the variational equality
\begin{equation}
\lambda_1(\varphi,V)=\inf_{g^2\in\Fcal}\big\{I_{\varphi}^{\ssup{\rm c}}(g^2)-(V,g^2)\big\}.
\end{equation}
For $V\in \Ccal_{\rm b}(G)$, introduce the operator $\Pcal_t^{\varphi,V}$ on $\ell^2(\alpha_tG\cap\Z^d)$ by
\begin{equation*}
\Pcal_t^{\varphi,V}f(z)=\E_z^{\beta_t^{-1}\varphi_t}\Big[\exp\Big\{-\frac{t}{\beta_t\alpha_t^2}\int_G V(y)L_t(y)\,\d y\Big\}\1\{X_{[0,t]}\subset\alpha_tG\}f(X_t)\Big].
\end{equation*}
Then, \eqref{eqn:rescaled-fix-cumulant-generating} is shown for all $V\in \Ccal_{\rm b}(G)$ if we verify
\begin{equation}\label{eqn:cumulant-generating}
\lim_{t\to\infty}\frac{\beta_t\alpha_t^2}{t}\log\Pcal_t^{\varphi,V}\1(0)=-\lambda_1(\varphi,V)
\end{equation}
for all such $V$ (including $V\equiv 0$). Recalling the notation \eqref{eqn:vtdef} and using that $L_t$ is a step function, we calculate
\begin{align*}
\Pcal_t^{\varphi,V}f(z)=\E_z^{\beta_t^{-1}\varphi_t}\Big[\exp\Big\{-\frac{1}{\beta_t\alpha_t^2}\int_0^t V_t(X_s)\,\d s\Big\}\1\{X_{[0,t]}\subset\alpha_tG\}f(X_t)\Big].
\end{align*}
Consequently, $\Pcal_t^{\varphi,V}$ admits the semigroup representation
$$
\Pcal_t^{\varphi, V}=\exp\{t(\Delta^{\beta_t^{-1}\varphi_t}-\beta_t^{-1}\alpha_t^{-2}V_t)\}=\exp\Big\{t\beta_t^{-1}\alpha_t^{-2}\big[\alpha_t^2\Delta^{\varphi_t}-V_t\big]\Big\},
$$
where the operator in the exponent is considered in $\ell^2(\alpha_tG\cap\Z^d)$ with zero boundary condition. Note that $\Pcal_t^{\varphi, V}$ has the same principal eigenfunction $v_t$ as the operator $-\alpha_t^2\Delta^{\varphi_t}+V_t$ has, and the corresponding principal eigenvalue is given by $\exp\big\{-\frac{t}{\beta_t\alpha_t^2}\lambda_1^{\ssup t}(\varphi,V)\big\}$. An eigenvalue expansion yields, for each $t\geq 0$,
\begin{equation*}
\exp\Big\{-\frac{t}{\beta_t\alpha_t^2}\lambda_1^{\ssup t}(\varphi,V)\Big\}\big(v_t(0)\big)^2\leq \Pcal_t^{\varphi,V}\1(0)\leq\vert\alpha_tG\vert^2\exp\Big\{-\frac{t}{\beta_t\alpha_t^2}\lambda_1^{\ssup t}(\varphi,V)\Big\}.
\end{equation*}
Thus, (\ref{eqn:cumulant-generating}) follows by Lemmas \ref{lem:eigenvalues-converge} and \ref{lem-eigenfunction-decay}.
\end{proof}

\begin{remark}\label{rem:ldp-rescaled-fix}
In the proof of the lower bound in Theorem~\ref{thm:ldp-growing-good}, we in fact use Proposition~\ref{lem:rescaled-fix-ldp} for the local times $L_t$ under $\P^{\beta_t^{-1}(\varphi_t-\delta\alpha_t^{-2})}$ where $0<\delta<m$ instead of $\P^{\beta_t^{-1}\varphi_t}$. It is easily seen that the proof given above works just as well with this slight modification as we are only subtracting a spatially constant factor that vanishes as $t\to\infty$. However, we prefer to omit this modification in the proof and in the statement of the lemma for reasons of conciseness.
\end{remark}

\subsection{Large deviations for rescaled conductances}\label{subsec:asycond}

In this section, we characterise the asymptotic probability of having a small conductance field. The first important lemma will be used for the lower bound in Theorem~\ref{thm:ldp-growing-good} and reads like the lower bound of an LDP for the rescaled conductances in a growing box. Consider the set
\begin{equation}\label{Atdef}
A(B,\psi,\delta)=\{\widetilde\psi\colon B\times \Ncal\to (0,\infty)\,\vert\, \psi-\delta\leq\widetilde\psi\leq \psi\}
\end{equation}
and recall the scale function $\beta_t\gg1$ from \eqref{betadef}. It turns out that we will need a lower estimate for the probability of the event that $\beta_ta$ is $\delta\alpha_t^{-2}$-close to $\varphi_t$ on $B_t=\alpha_t G\cap\Z^d$, i.e., lies in $A(B_t,\varphi_t,\delta\alpha_t^{-2})$. Here, $\varphi_t$ is the unscaled version of $\varphi$ defined in \eqref{varphitdef}.

\begin{lemma}\label{lem:lbound-environment} Fix a scale function $\beta_t\gg 1$, positive numbers $m<M$ and some $\varphi\colon G\times\Ncal\to(m,M)$ such that $\varphi(\cdot,e)\in\Ccal_{\rm b}(G)$ for any $e\in\Ncal$. Then, for any $\delta\in(0,m)$,
\begin{equation}\label{ldp-aux-conductances}
\liminf_{t\to\infty}\frac 1{\beta_t^\eta\alpha_t^d}\log\Pr\big(\beta_ta\in A(B_t,\varphi_t,\delta\alpha_t^{-2})\big)\geq -D\sum_{e\in\Ncal}\int_G \varphi(y,e)^{-\eta}\,\d y.
\end{equation}
\end{lemma}

\begin{proof}As a pre-step we first derive this lower estimate for the event that $\beta_ta$ is only $\delta$-close, i.e., we prove \eqref{ldp-aux-conductances} with $\delta\alpha_t^{-2}$ replaced by $\delta$. Assumption~\ref{assumption:tails} yields the existence of a non-decreasing map $R\colon [0,\infty)\rightarrow[0,\infty)$ with $R(\eps)\stackrel{\eps\to 0}{\rightarrow}0$ such that, for all $\eps>0$,  
$$
-D\eps^{-\eta}(1+R(\eps))\leq\log\Pr(a(0,e_1)\leq\eps)\leq-D\eps^{-\eta}(1-R(\eps)).
$$ 
Therefore, we may estimate
\begin{align}\label{eqn:ldp-conductances-1}
\Pr(\beta_ta\in A(B_t,\varphi_t,\delta))&=\prod_{z,e}\Big[\Pr(a(z,e)\leq\beta_t^{-1}\varphi_t(z,e))-\Pr(a(z,e)\leq\beta_t^{-1}(\varphi_t(z,e)-\delta))\Big]\notag\\
&\geq \prod_{z,e}\Big[\e^{-D\beta_t^\eta \varphi_t(z,e)^{-\eta}(1+R(\beta_t^{-1}M))}-\e^{-D\beta_t^\eta(\varphi_t(z,e)-\delta)^{-\eta}(1-R(\beta_t^{-1}M))}\Big]\notag\\
&=\prod_{z,e}\e^{-D\beta_t^\eta \varphi_t(z,e)^{-\eta}(1+R(\beta_t^{-1}M))}\notag\\
&\quad\times\prod_{z,e}\Big[1-\e^{-D\beta_t^\eta\big[(\varphi_t(z,e)-\delta)^{-\eta}(1-R(\beta_t^{-1}M))-\varphi_t(z,e)^{-\eta}(1+R(\beta_t^{-1}M))\big]}\Big].
\end{align}
Pick some positive $\delta_0$ and choose $t$ large enough to satisfy 
$$
\Big(\frac{M}{M-\delta}\Big)^\eta>\frac{1+R(\beta_t^{-1}M)}{1-R(\beta_t^{-1}M)} + \delta_0.
$$ 
Thus, for all $z\in B_t$, 
$$
(\varphi_t(z,e)-\delta)^{-\eta}(1-R(\beta_t^{-1}M))-\varphi_t(z,e)^{-\eta}(1+R(\beta_t^{-1}M))>2\delta_0M^{-1}.
$$ 
We may therefore continue \eqref{eqn:ldp-conductances-1} by
\begin{align*}
\log\Pr(\beta_ta\in A(B_t,\varphi_t,\delta))&\geq-D\beta_t^\eta\sum_{z,e}\varphi_t(z,e)^{-\eta}(1+R(\beta_t^{-1}M))
&+d\vert\alpha_tG\vert\log\big(1-\e^{-2DM^{-1}\delta_0\beta_t^\eta}\big).
\end{align*}
Finally, by H\"older's reverse inequality and merging asymptotically negligible terms,
\begin{align*}
\frac 1{\beta_t^\eta\alpha_t^d}&\log\Pr(\beta_ta\in A(B_t,\varphi_t,\delta))\geq-D\alpha_t^{-d}\sum_{z,e} \varphi_t(z,e)^{-\eta}+o(1)\\
&=-D\alpha_t^{-d}\sum_{z,e} \Big(\int_{[0,1]^d}\varphi\Big(\frac{z+y}{\alpha_t},e\Big)\Big)^{-\eta}\,\d y+o(1)\geq-D\sum_e\int_G \varphi(y,e)^{-\eta}\,\d y+o(1).
\end{align*}
Now we prove \eqref{ldp-aux-conductances}. To estimate the asymptotic probability of the event $\beta_ta\in A(B_t,\varphi_t,\delta\alpha_t^{-2})$ instead of $A(B_t,\varphi_t,\delta)$, we need the additional technical condition on the existence of an increasing density for small conductances, which we put in Theorem~\ref{thm:ldp-growing-good}. Under this assumption, we may easily estimate for some $C\in(0,\infty)$, any $x\in(m,M)$ and all sufficiently large $t$,
$$
\Pr\big(x-\alpha_t^{-2}\delta\leq a(z,e)\leq x\big)\geq \frac {C}{\alpha_t^2}\Pr\big(x-\delta\leq a(z,e)\leq x\big)
$$ 
Using this in what we proved so far, i.e., in \eqref{ldp-aux-conductances}  with $\delta\alpha_t^{-2}$ replaced by $\delta$, we obtain
$$
\log\Pr\big(\beta_ta\in A(B_t,\varphi_t,\delta\alpha_t^{-2})\big)\geq \log\big(\sfrac 1C\alpha_t^{-2d\vert\alpha_t G\vert}\big)+\log\Pr\big(\beta_ta\in A(B_t,\varphi_t,\delta)\big).
$$ 
Since obviously $\log\big(\alpha_t^{-2d\vert\alpha_tG\vert}\big)=o(\beta_t^\eta\alpha_t^d)$, we arrive at the desired result.
\end{proof}

For the proof of the upper bound in Theorem~\ref{thm:ldp-growing-good} in Section~\ref{sec-proofThmgoodupper} below, we will need also a large-deviations statement about the rate function of the conductances, applied to the rescaled conductances themselves. Recall the rescaled conductance field $a_t(y,e)=\beta_t a(\lfloor\alpha_ty\rfloor,e)$ from \eqref{omegatdef} for $y\in G,\,e\in\Ncal$.

\begin{lemma}\label{lem:ubound-environment}
Fix some scale function $\beta_t\gg 1$. Then, for any $\eps>0$, we have
\begin{equation*}
\limsup_{t\to\infty}\frac 1{\alpha_t^d\beta_t^\eta}\log\Pr\Big(\sum_{e\in\Ncal}\int_G(a_t(y,e))^{-\eta}\,\d y\geq \eps\Big)\leq -D\eps.
\end{equation*}
\end{lemma}

\begin{proof}
Choose some positive $x<D$. By the exponential Chebychev inequality, 
\begin{equation*}
\Pr\Big(\sum_e\int_G(a_t(y,e))^{-\eta}\,\d y\geq \eps\Big)\leq\e^{-\alpha_t^d\beta_t^\eta x\eps}\Big\langle\exp\Big\{\alpha_t^d\beta_t^\eta x\sum_e\int_G(a_t(y,e))^{-\eta}\,\d y\Big\} \Big\rangle.
\end{equation*}
Therefore, it will be sufficient to show that
\begin{equation}
\limsup_{t\to\infty}\frac 1{\alpha_t^d\beta_t^\eta}\log\Big\langle\exp\Big\{\alpha_t^d\beta_t^\eta x\sum_e\int_G(a_t(y,e))^{-\eta}\,\d y\Big\} \Big\rangle\leq 0.
\end{equation}
We make use of the independence of conductances over edges and obtain after rescaling
\begin{align*}
\Big\langle\exp\Big\{\alpha_t^d\beta_t^\eta x\sum_e\int_G(a_t(y,e))^{-\eta}\,\d y\Big\}\Big\rangle&\leq\Big\langle\exp\Big\{\beta_t^\eta x\sum_{e}\sum_{z\in\alpha_tG\cap\Z^d}(\beta_ta(z,e))^{-\eta}\Big\}\Big\rangle \leq\Big\langle\e^{xa^{-\eta}}\Big\rangle^{C\alpha_t^d}
\end{align*}
for some constant $C>0$ with $a=a(0,e_1)$ representing a single conductance. Consequently, it will now be sufficient to show that $\langle\e^{xa^{-\eta}}\rangle<\infty$ for $x<D$. This is implied by Assumption \ref{assumption:tails}. Indeed, with some bounded residual term $r$ such that $r(s)\to 0$ as $s\to\infty$,
\begin{align*}
\langle\e^{xa^{-\eta}}\rangle&=\int_0^\infty\Pr\big(\e^{xa^{-\eta}}>s\big)\,\d s\leq b+\int_b^\infty\Pr\big(a<(\log s)^{-1/\eta}x^{1/\eta}\big)\,\d s\\
&=b+\int_b^\infty\exp\big\{-(D/x)(\log s)[1+r(s)]\big\}\,\d s
\end{align*}
for arbitrary $b>0$. Choosing $b$ so large that $(D/x)[1+r(s)]>c$ for all $s>b$ and some $c>1$, we arrive at
$\langle\e^{xa^{-\eta}}\rangle\leq b+\int_b^\infty s^{-c}\,\d s<\infty$.
\end{proof}

\section{Proof of Theorem \ref{thm:ldp-growing-good}}\label{sec:ldp-growing-proofs}

In this section, we assemble the results from the previous sections and prove Theorem \ref{thm:ldp-growing-good}. 
Recall that we are working on the space $\Fcal=\{f^2\colon f\in L^2(G),\,\Vert f\Vert_2=1\}$, equipped with the weak topology of integrals against bounded continuous functions $G\to\R$.

\subsection{Compactness of the level sets of $\boldsymbol{J^{\ssup{\rm c}}}$.}\label{subsec:ldp-growing}

\noindent Let us show that the level sets 
$$
I_s=\{f^2\in\Fcal\colon J^{\ssup{\rm c}}(f^2)\leq s\},\qquad s\in[0,\infty),
$$
of $J^{\ssup{\rm c}}$ are compact. To that end, choose $s\geq 0$ and some sequence $(f_n)_{n\in\N}$ in $I_s$. Abbreviate $p=\frac{2\eta}{\eta+1}$. We need to show the existence of some $f\in I_s$ such that, along some subsequence, 
\begin{equation}\label{eqn-WeakConvergence}
\int_G f_n^2(y)V(y)\,\d y\rightarrow \int_G f(y)^2V(y)\,\d y\qquad\text{ as }t\to\infty
\end{equation}
for all $V\colon G\to\R$ bounded and continuous. As $\Fcal$ is bounded in $L^2(G)$, the Banach-Alaoglu theorem implies that there exists $f\in L^2(G)$ such that 
\begin{equation}
\int_G f_n(y)V(y)\,\d y\rightarrow \int_G f(y)V(y)\,\d y\qquad\text{ as }t\to\infty
\end{equation}
for all $V\in L^2(G)$, after choosing a subsequence. By H\"older's inequality and boundedness of the test functions, this implies \eqref{eqn-WeakConvergence} for some subsequence. Thus, it remains to show that $f\in I_s$. For the requirement that $\Vert f\Vert_2=1$, it is necessary to show convergence of $f_n$ in the strong $L^2(G)$-sense. This is implied by the Rellich-Kondrashov theorem (Theorem \ref{thm:rellich-kondrashov}) in analogy with Section \ref{sec:varprobs}. At this point, we need the restrictions on the parameter $\eta$ made in Theorem \ref{thm:ldp-growing-good} ($\eta>d/2$ and if $d=1$, $\eta\geq 1$).

The requirement that $J^{\ssup{\rm c}}(f^2)\leq s$ still needs to be verified. Let $i\in\{1,\ldots,d\}$. As the sequence $(\partial_if_n)_{n\in\N}$ is bounded in $L^p(G)$, we may assume that it converges weakly (that is, with respect to integrals against functions $V\in L^q(G)$ where $1/p+1/q=1$) against some $g_i\in L^p(G)$. As all norms are lower semicontinuous with respect to the weak topology, we have $\sum_{i=1}^d\Vert g_i\Vert_p^p\leq s$. Since $J^{\ssup{\rm c}}(f^2)=\sum_{i=1}^d\Vert\partial_if\Vert_p^p$, the assertion is shown if only $\partial_if=g_i$ for all $i\in\{1,\ldots,d\}$. In order to show this, choose some $V\in \Ccal^\infty_0(G)\subset L^q(G)$. On the one hand, 
$$
\int_G \partial_if_n(y)V(y)\,\d y\underset{n\to\infty}\rightarrow \int_G g_i(y)V(y)\,\d y.
$$ 
On the other hand, 
$$
\int_G f_n(y)\partial_i V(y)\,\d y\underset{n\to\infty}\rightarrow \int_G f(y)\partial_iV(y)\,\d y
$$ 
as $\partial_iV\in L^2(G)$ and $f_n\to f$ weakly in $L^2(G)$. The limits above imply (by the definition of the weak derivative) 
$$
\int_G g_i(y)V(y)\,\d y=\int_G \partial_if(y)V(y)\,\d y
$$ 
for all $V\in \Ccal^\infty_0(G)$. This shows $\partial_if=g_i$ for all $i\in\{1,\ldots,d\}$ as both functions are elements of $L^p(G)$, and $\Ccal^\infty_0(G)$ is dense in $L^q(G)$. This means the level sets $I_s$ of $J^{\ssup{\rm c}}$, and therefore those of $J^{\ssup{\rm c}}_0$, are compact.

\subsection{Proof of Theorem \ref{thm:ldp-growing-good}, lower bound.}

Let us go on with the proof of the lower bound. We start by recalling an auxiliary result from \cite{KSW11}. It ensures a certain continuity property of probabilities of certain events with regard to small changes of the conductances. 

\begin{lemma}\label{cor:LDP_density_estimate}
Let $\varphi,\psi\colon \Z^d\times\Ncal\to(0,\infty)$ with $0<\psi(x,e)-\eps\leq\varphi(x,e)\leq\psi(x,e)+\eps$ for some $\eps>0$ and all $x\in\Z^d$ and $e\in\Ncal$. Moreover, let $F$ be some event that depends on the process $(X_s)_{s\in[0,t]}$ up to time $t$ only. Then
$$
\P_0^\varphi\big(F\big)\geq\e^{-4d\eps t}\P_0^{\psi-\eps}\big(F\big).
$$
\end{lemma}

With this tool at hand, we now turn to the proof of the lower bound in Theorem \ref{thm:ldp-growing-good}. Choose an open set $\Ocal$ in $\Fcal$ with respect to the weak topology and some function $f^2\in \Ocal$. Our goal is to prove \eqref{eqn:dvg-ldp-lowerLarge}. We will write just $\{L_t\in\Ocal\}$ for $ \{L_t\in\Ocal,\supp(\ell_t)\subset\alpha_t G\}$.

We may assume that $f^2\in H_0^1(G)\cap \Ocal$ since \eqref{eqn:dvg-ldp-lowerLarge} is trivial otherwise. For the same reason, it is possible to assume that $f\in W^{1,p}(G)$ with $p=\frac{2\eta}{\eta+1}$. By convolution with an appropriate mollifier and norming, we consequently obtain functions $f_\eps\in \Ccal^1_0(G)$ such that $f_\eps\to f$ as $\eps\searrow 0$ both in $H^1_0(G)$ and in $W^{1,p}(G)$. As $\Ocal$ is open in the weak $L^2$-topology, it is also open in the strong $L^2$-topology and therefore $f_\eps\in \Ocal$ for $\eps$ small enough. Let us fix such an $\eps>0$ and some $M>0$ and define
\begin{equation*}
\varphi_M^{(f,\eps)}(y,e)=M^{-1}\vee(D\eta)^{\frac 1{\eta+1}}\vert\partial_ef_\eps(y)\vert^{-\frac 2{\eta+1}} \wedge M
\end{equation*}
with the convention $0^{-\frac 2{\eta+1}}=\infty$. Note that this function is continuous in the first argument. In analogy with Section~\ref{sec:ldp-growing-prelim}, put 
$$
\varphi_t(z,e)=\int_{[0,1]^d}\varphi_M^{(f,\eps)}\Big(\frac{z+y}{\alpha_t},e\Big)\,\d y,\quad z\in B_t, e\in\Ncal.
$$
Choose some $\delta\in(0,M^{-1})$ and $\beta_t$ such that $\beta_t^\eta\alpha_t^d=t\beta_t^{-1}\alpha_t^{-2}$ (the condition $\alpha_t\ll t^{\frac d{d+2}}$ ensures $\beta_t\gg 1$). We restrict the expectation with respect to the conductances to the event where $\beta_ta$ lies in $A_t=A(B_t, \varphi_t,\delta\alpha_t^{-2})$, where we recall \eqref{Atdef}. We estimate
\begin{align}\label{eqn:box-lowerbound1}
\langle\P_0^a(L_t\in \Ocal)\rangle&\geq \langle\P_0^a(L_t\in \Ocal)\1_{\{\beta_ta\in A_t\}}\rangle\geq \inf_{\psi\in A_t}\P_0^{\beta_t^{-1}\psi}(L_t\in \Ocal)\Pr(\beta_ta\in A_t)\notag\\
&\geq \e^{-4dt\delta\alpha_t^{-2}\beta_t^{-1}}\P_0^{\beta_t^{-1}(\varphi_t-\delta\alpha_t^{-2})}(L_t\in \Ocal)\Pr(\beta_ta\in A_t),
\end{align}
where the last step is due to Lemma~\ref{cor:LDP_density_estimate}. Now, by Proposition~\ref{lem:rescaled-fix-ldp} (taking Remark~\ref{rem:ldp-rescaled-fix} into consideration) and Lemma~\ref{lem:lbound-environment}, we obtain (with our particular choice of $\beta_t$)
\begin{equation*}
\liminf_{t\to\infty}t^{-\frac\eta{1+\eta}}\alpha_t^{-\frac{d-2\eta}{1+\eta}}\log \langle\P_0^a(L_t\in \Ocal)\rangle\geq -\sum_{e}\int_G\Big(\varphi_M^{(f,\eps)}(y,e)\big(\partial_ef_\eps(y)\big)^2+D\varphi_M^{(f,\eps)}(y,e)^{-\eta}\Big)\,\d y-4d\delta.
\end{equation*}
As $\delta$ was chosen arbitrarily small, we may omit the last term in the above inequality. Moreover, the resulting scale is seen to be equal to $\gamma_t$ from Theorem \ref{thm:ldp-growing-good}. Then, it is quickly verified that
\begin{equation*}
\sum_{e}\int_G\Big(\varphi_M^{(f,\eps)}(y,e)\big(\partial_ef_\eps(y)\big)^2+D\varphi_M^{(f,\eps)}(y,e)^{-\eta}\Big)\,\d y\rightarrow J^{\ssup{\rm c}}\big(f_\eps^2\big)
\end{equation*}
as $M\to\infty$ by applying the monotone and dominated convergence theorems on the parts of the integral where $\partial_ef_\eps$ is equal to $0$, between $0$ and $1$ and greater than $1$, respectively. Since $M$ was chosen arbitrarily,
\begin{equation*}
\liminf_{t\to\infty}\frac1{\gamma_t}\log \langle\P_0^a(L_t\in \Ocal)\rangle\geq J^{\ssup{\rm c}}\big(f_\eps^2\big).
\end{equation*}
Letting $\eps\searrow 0$, we may also conclude
\begin{equation*}
\liminf_{t\to\infty}\frac1{\gamma_t}\log \langle\P_0^a(L_t\in \Ocal)\rangle\geq J^{\ssup{\rm c}}(f^2)
\end{equation*}
as $\partial_ef_\eps\to\partial_ef$ in the $L^p$-norm. We arrive at the desired lower bound by taking the infimum over all functions $f\in H_0^1(G)\cap \Ocal$ remembering that $f$ was chosen arbitrarily in $\Ocal$.

\subsection{Proof of Theorem \ref{thm:ldp-growing-good}, upper bound.}\label{sec-proofThmgoodupper}

\noindent Let us now turn to the proof of the upper bound. Let $\Ccal$ be a closed set of probability densities on $G$. We will show that \eqref{eqn:dvg-ldp-upperLarge} holds, even when we replace the starting point $0$ by any other site $x\in B_t=\alpha_t G\cap \Z^d$, possibly depending on $t$, uniformly in $x$. Note that $L_t\in\Ccal$ is equivalent to $\frac 1t\ell_t\in\Ccal_t$, where 
\begin{equation}\label{rescaledC}
\Ccal_t=\{g^2\colon g\in\ell^2(B_t),\Vert g\Vert=1,\alpha_t^dg^2(\lfloor\alpha_t\cdot\rfloor)\in\Ccal\}
\end{equation}
is the set of rescalings of step functions in $\Ccal$. We now fix any starting point $x\in B_t=\alpha_t G\cap \Z^d$ and estimate the probability term with the help of \cite[Theorem 3.6]{BHK07}, which states that
\begin{equation}\label{eqn:uboundnonexit1}
\P^a_x\big(L_t\in\Ccal,\supp(\ell_t)\subset\alpha_tG\big)
=\P_x^a\big(\smfrac 1t\ell_t\in\Ccal_t,\supp(\ell_t)\subset B_t\big)
\leq \exp\Big\{-t \inf_{\mu\in\Ccal_t}\Lambda_a(B_t,\mu)\Big\} \e^{C_t},
\end{equation}
where we put
$$
\Lambda_a(B_t,\mu)=\sum_{x,y\in B_t\colon x\sim y}a_{x,y}\big(\sqrt{\mu(x)}-\sqrt{\mu(y)}\big)^2.
$$
Furthermore, $C_t$ is an error term that can be estimated as follows.
$$
C_t= |B_t| \log\big(\eta_{B_t}\sqrt{8\e} t)+\log|B_t|+\frac{|B_t|}{4t},
$$
where 
$$
\eta_{B_t}=\max\Big\{\max_{x\in B_t}\sum_{y\in B_t\setminus \{x\}} |\Delta^a_{x,y}|,\max_{y\in B_t}\sum_{x\in B_t\setminus \{y\}} |\Delta^a_{x,y}|,1\Big\}
$$ 
is bounded in $t$, since the conductances are, according to our assumptions. Furthermore, from our upper bound on $\alpha_t$ in Theorem~\ref{thm:ldp-growing-good}, we have that $\log t\ll\beta^{\eta}$; see \eqref{betadef}. This shows that the error term $C_t$ is negligible on the scale $\gamma_t=\alpha_t^d\beta_t^\eta$; see \eqref{gammatdef}.

Now we use H\"older's inequality to estimate, for $g^2=\mu\in\Ccal_t$ having support in $B_t$,
\begin{equation}\label{Hoelder}
\Lambda_a(B_t,\mu)=\sum_{e\in\Ncal}\sum_{z\in\Z^d}a(z,e)|g(z+e)-g(z)|^2
\geq \Big(\sum_{z,e}|g(z+e)-g(z)|^{\frac{2\eta}{\eta+1}}\Big)^{(\eta+1)/\eta}\Big(\sum_{z,e}\big(a(z,e)\big)^{-\eta}\Big)^{-1/\eta}.
\end{equation}
Recalling the rescaled conductance field $a_t(y,e)=\beta_t a(\lfloor y\alpha_t\rfloor, e)$ from \eqref{omegatdef} and introducing the notation
\begin{equation}\label{chiBCdef}
\chi^{\ssup{\rm d}}(B_t,\Ccal_t)=\inf_{g^2\in\Ccal_t}\sum_{e\in\Ncal}\sum_{z\in\Z^d}|g(z+e)-g(z)|^{\frac{2\eta}{\eta+1}},
\end{equation}
we see that
$$
\inf_{\mu\in\Ccal_t}\Lambda_a(B_t,\mu)\geq \frac 1{\beta_t\,\alpha_t^2}\Big(\alpha_t^{\frac{2\eta-d}{\eta+1}}\chi^{\ssup{\rm d}}(B_t,\Ccal_t)\Big)^{(\eta+1)/\eta}\Big(\sum_e\int_G\big(a_t(y,e)\big)^{-\eta}\,\d y\Big)^{-1/\eta}.
$$
Pick some small $\delta>0$. By Lemma \ref{lem:tilted-rescaling} below, we have, for all $t$ large enough,
\begin{equation}\label{eqn:uboundnonexit2}
\inf_{\mu\in\Ccal_t}\Lambda_a(B_t,\mu)\geq \frac1{\beta_t\alpha_t^{2}}\big[\chi^{\ssup{\rm c}}(G,\Ccal)-\delta\big]^{(\eta+1)/\eta}\Big(\sum_e\int_G\big(a_t(y,e)\big)^{-\eta}\,\d y\Big)^{-1/\eta}.
\end{equation}
Choose now a large positive number $M$ and some small $\eps>0$ and define on the environment space of measurable non-negative functions $G\times\Ncal\to(0,\infty)$, the events
\begin{eqnarray}
A_n&=&\Big\{\varphi\colon \sum_e\int_G\varphi(y,e)^{-\eta}\,\d y \in ((n-1)\eps,n\eps]\Big\},\qquad n\in\N, n\leq M/\eps,\label{Asetdef}\\
B_1&=&\Big\{\varphi\colon\sum_e\int_G\varphi(y,e)^{-\eta}\,\d y\geq M\Big\}
\qquad\mbox{and}\qquad
B_2=\Big\{\varphi\colon\sum_e\int_G\varphi(y,e)^{-\eta}\,\d y\leq\eps\Big\}.\label{Bsetsdef}
\end{eqnarray}
We proceed by combining \eqref{eqn:uboundnonexit1} and \eqref{eqn:uboundnonexit2} and splitting the expectation w.r.t.~the environment as
$$
\begin{aligned}
\P^a_0\big(L_t\in\Ccal,\supp(\ell_t)\subset\alpha_tG\big)
&\leq \Pr(a_t\in B_1) +\sum_{n=1}^{M/\eps}\e^{-t\beta_t^{-1}\alpha_t^{-2}\big[\chi^{\ssup{\rm c}}(G,\Ccal)-\delta\big]^{(\eta+1)/\eta}(n\eps)^{-1/\eta}}\Pr(a_t\in A_n)\\
&\quad+\e^{-t\beta_t^{-1}\alpha_t^{-2}\big[\chi^{\ssup{\rm c}}(G,\Ccal)-\delta\big]^{(\eta+1)/\eta}\eps^{-1/\eta}}\Pr(a_t\in B_2).
\end{aligned}
$$
For the environment terms, we use Lemma \ref{lem:ubound-environment} to calculate their asymptotic behavior, noting that $t\beta_t^{-1}\alpha_t^{-2}=\beta_t^\eta\alpha_t^d,
$, by the choice of $\beta_t$ in \eqref{betadef}. The condition $\alpha_t\ll t^{\frac \eta{d(\eta+1)}}$ ensures that $\beta_t\gg 1$. Noting the definition of $\gamma_t$ in \eqref{gammatdef}, this means that
$$
\begin{aligned}
\limsup_{t\to\infty}&\gamma_t^{-1}\log\langle\P^a_0\big(L_t\in\Ccal,\supp(\ell_t)\subset\alpha_tG\big)\rangle\\
&\leq -DM \, \vee\max_n\Big[ -\big[\chi^{\ssup{\rm c}}(G)-\delta\big]^{\frac{\eta+1}\eta}(n\eps)^{-\frac 1\eta}-D((n-1)\eps)\Big]\vee -\big[\chi^{\ssup{\rm c}}(G,\Ccal)-\delta\big]^{\frac{\eta+1}\eta}\eps^{-\frac1\eta}\\
&\leq -DM \, \vee
\sup_{y\in(\eps,M)} \Big[-\big[\chi^{\ssup{\rm c}}(G)-\delta\big]^{\frac{\eta+1}\eta}y^{-\frac1\eta}-Dy\Big]+D\eps\vee -\big[\chi^{\ssup{\rm c}}(G)-\delta\big]^{\frac{\eta+1}\eta}\eps^{-\frac1\eta}.
\end{aligned}
$$
Optimizing over $y$ after choosing $M$ large enough and $\eps$ small enough, and finally taking limits $\delta\to 0$ and $\eps\to 0$, yields the desired result.

\begin{lemma}\label{lem:tilted-rescaling}
Let $\eta>d/2$. Fix a closed subset $\Ccal$ of $\Fcal$ with rescaled version $\Ccal_t$ defined in \eqref{rescaledC}. Then we have 
$$
\liminf_{t\to\infty}\alpha_t^{\frac{2\eta-d}{\eta+1}}\chi^{\ssup{\rm d}}(B_t,\Ccal_t)\geq \chi^{\ssup{\rm c}}(G,\Ccal).
$$
\end{lemma}

\begin{proof}
We may assume that $\Ccal_t$ is nonempty. Pick minimisers $g_t\in\Ccal_t$ of the formula for $\chi^{\ssup{\rm d}}(B_t,\Ccal_t)$ in \eqref{chiBCdef} such that 
\begin{equation}
\chi^{\ssup{\rm d}}(B_t,\Ccal_t)=\sum_{e\in\Ncal}\sum_{z\in\Z^d}\vert g_t(z+e)-g_t(z)\vert^{\frac{2\eta}{\eta+1}}.
\end{equation}
Let us consider the rescaled versions $\widetilde f_t\in L^2(G)$ defined as 
$$
\widetilde f_t(y)=\alpha_t^{d/2} g_t (\lfloor \alpha_t y \rfloor).
$$ 
Note that $\widetilde f_t\in\Ccal$. Due to norming of the sequence $\widetilde f_t$ and closedness of $\Ccal$, we find $f\in\Ccal$ such that $\widetilde f_t\to f$ in the weak $L^2$-sense, which in turn implies convergence in the weak topology we are considering. Let us show that
$$
\liminf_{t\to\infty}\alpha_t^{\frac{2\eta-d}{\eta+1}}\chi^{\ssup{\rm d}}(B_t,\Ccal_t)\geq \sum_{e\in\Ncal}\int_{\R^d}\vert\partial_e f(y)\vert^{\frac{2\eta}{\eta+1}}\,\d y,
$$
which instantly yields the desired result. In analogy with the construction in Lemma \ref{lem:eigenvalues-converge}, we find functions $f_t\in H_0^1(G)$ (trivially extended to $\R^d$) such that for almost all $y\in G$, $e\in\Ncal$ and $t>0$,
\begin{equation}
\partial_e f_t(y)=\alpha_t^{1+d/2}\big[g_t(\lfloor\alpha_ty\rfloor+e)-g_t(\lfloor\alpha_ty\rfloor)\big].
\end{equation}
In particular, $\partial_ef_t$ is almost everywhere constant on the boxes $\alpha_t^{-1}(z+[0,1]^d)$ with $z\in\Z^d$, thus
$$
\begin{aligned}
\alpha_t^{\frac{2\eta-d}{\eta+1}}\sum_{e\in\Ncal}\sum_{z\in\Z^d} 
\vert g_t(z+e)-g_t(z)\vert^{\frac{2\eta}{\eta+1}}
&=\alpha_t^{\frac{2\eta-d}{\eta+1}}\alpha_t^d\sum_{e\in\Ncal}\int_{\R^d}\big(\alpha_t^{-1-d/2}\vert\partial_e f_t(y)\vert\big)^{\frac{2\eta}{\eta+1}}\,\d y\\
&=\sum_{e\in\Ncal}\int_{\R^d}\vert\partial_e f_t(y)\vert^{\frac{2\eta}{\eta+1}}\,\d y.
\end{aligned}
$$
It therefore remains to show that
\begin{equation}\label{eqn:tilted-rescaling-0}
\liminf_{t\to\infty}\sum_{e\in\Ncal}\int_{\R^d}\vert\partial_e f_t(y)\vert^{\frac{2\eta}{\eta+1}}\,\d y\geq
\sum_{e\in\Ncal}\int_{\R^d}\vert\partial_e f(y)\vert^{\frac{2\eta}{\eta+1}}\,\d y.
\end{equation}
To that end, we need to establish weak convergence of the $f_t$ towards $f$ and convergence to $1$ of their $L^2$-norms. Then, \eqref{eqn:tilted-rescaling-0} follows from lower semicontinuity of the functional $f^2\mapsto \sum_{e\in\Ncal}\|\partial_e f\|_p^p$ (with $p=\frac{2\eta}{1+\eta}$), which follows from the compactness of the level sets of  $J^{\ssup{\rm c}}$. Here, we require the assumptions made on the value of $\eta$. According to \eqref{eqn:interpolation-residual}, we obtain the desired convergence properties and even $\Vert f_t-\widetilde f_t\Vert\to 0$ if
\begin{equation}\label{eqn:tilted-rescaling-1}
\sum_{e\in\Ncal}\sum_{z\in\Z^d} 
\vert g_t(z+e)-g_t(z)\vert^2\to 0\quad\text{as }t\to\infty.
\end{equation}

As we are on a discrete space and consider normed functions, we may estimate
\begin{equation}\label{eqn:tilted-rescaling-2}
\sum_{e\in\Ncal}\sum_{z\in\Z^d} \vert g_t(z+e)-g_t(z)\vert^2
\leq C\sum_{e\in\Ncal}\sum_{z\in\Z^d} \vert g_t(z+e)-g_t(z)\vert^{\frac{2\eta}{\eta+1}}
=C\chi^{\ssup{\rm d}}(B_t),
\end{equation}
for some $C>0$. As $G$ is open, it contains the box $[-\delta,\delta]^d$ with some $\delta>0$. With $Q_t=\alpha_t[-\delta,\delta]^d\cap\Z^d$, we have $\chi^{\ssup{\rm d}}(B_t)\leq\chi^{\ssup{\rm d}}(Q_t)$. By Lemma~\ref{lem:varprob-disc-not-solvable}, the latter vanishes as $t\to\infty$. Hence, \eqref{eqn:tilted-rescaling-2} implies \eqref{eqn:tilted-rescaling-1}, and the proof of Lemma~\ref{lem:tilted-rescaling} is complete.
\end{proof}

\section{Proof of Theorem \ref{thm:ldp-growing-bad}}\label{sec-proofThmBad}

Let us turn to the case where $\eta\leq d/2$. 

\subsection{Non-compactness of levels sets of $\boldsymbol{J^{\ssup{\rm c}}}$}

\noindent We start by showing that the level sets fail to be compact in this case. This property seems obvious after studying the variational problems in Section \ref{sec:varprobs}, but let us provide a rigorous proof.

\begin{lemma}\label{lem:level-sets-not-closed}
If $\eta\leq d/2$, the level sets of $J^{\ssup{\rm c}}$ are not closed. In particular, they are not compact.
\end{lemma}
\begin{proof}
From Lemma \ref{lem:varprob-cont-not-solvable}, we obtain sequences $(f_n)$ with $f_n\in H_0^1(G)$, $\Vert f_n\Vert_2\equiv 1$ for $n\in\N$ and $J^{\ssup{\rm c}}(f_n^2)\to 0$ as $n\to\infty$. In particular, $f_n^2\in\Fcal$ and for each level set $I_s=\{f^2\colon J^{\ssup{\rm c}}(f^2)\leq s\}$, $s>0$, we have $f_n^2\in I_s$ for $n$ large enough. As the sequence $(f_n)$ is bounded in $L^2$, there exists a weak limit $f$. We easily check by H\"older's inequality that
\begin{equation*}
\int_G f_n^2(y)V(y)\,\d y\rightarrow \int_G f(y)^2V(y)\,\d y\qquad\text{ as }t\to\infty
\end{equation*}
for all bounded and continuous $V\colon G\to \R$, so $(f_n)$ converges in the right topology. By lower semicontinuity of norms with regard to weak convergence, $J^{\ssup{\rm c}}(f^2)=0$. That implies $\Vert f\Vert_2=0$ which in turn yields $f^2\notin\Fcal$. As in particular $f^2\notin I_s$, the assertion follows.
\end{proof}

\subsection{Proof of Theorem \ref{thm:ldp-growing-bad}, upper bound}

\noindent 
Now, we proceed by showing the main statement, that is,
$$
\limsup_{t\to\infty}t^{-\frac\eta{\eta+1}}\log\langle\P^a_0\big(\supp(\ell_t)\subset\alpha_tG\big)\rangle\leq -K_{\eta,D}\,\chi^{\ssup{\rm d}}(\Z^d).
$$ 
Using a spectral Fourier expansion and estimating in standard way, we estimate the probability term as 
\begin{equation}\label{eqn:uboundnonexit3}
\P^a_0\big(\supp(\ell_t)\subset\alpha_tG\big)\leq \vert\alpha_tG\vert^2\exp\{-t\lambda_1^{\ssup t}(a)\},
\end{equation}
where $\lambda_1^{\ssup t}(a)$ is the principal eigenvalue of the operator $\Delta^a$ in the box $B_t$ with zero boundary condition. Using its Rayleigh-Ritz representation and H\"older's inequality analogously to \eqref{Hoelder}, we see that
\begin{align}\label{eqn:Hoelder-eigenvalue2}
\lambda_1^{\ssup t}(a)&\geq\beta_t^{-1}\inf_g \Big(\sum_{z,e}|g(z+e)-g(z)|^{\frac{2\eta}{\eta+1}}\Big)^{(\eta+1)/\eta}\Big(\sum_{z,e}\big(\beta_ta(z,e)\big)^{-\eta}\Big)^{-1/\eta}\notag\\
&=\beta_t^{-1}\alpha_t^{-\frac d\eta}(\chi^{\ssup{\rm d}}(B_t))^{(\eta+1)/\eta}\Big(\sum_e\int_G\big(a_t(y,e)\big)^{-\eta}\,\d y\Big)^{-1/\eta}.
\end{align}
In contrast to the proof of the upper bound in Theorem~\ref{thm:ldp-growing-good}, we continue the inequality differently by just estimating  $\chi^{\ssup{\rm d}}(B_t)\geq\chi^{\ssup{\rm d}}(\Z^d)$. Choose now a large positive number $M$ and some small $\eps>0$ and consider the events $A_n$, $B_1$ and $B_2$ defined in \eqref{Asetdef} and \eqref{Bsetsdef}. We proceed by combining \eqref{eqn:uboundnonexit3} and \eqref{eqn:Hoelder-eigenvalue2} and splitting the expectation w.r.t. the environment as
$$
\begin{aligned}
\vert\alpha_tG\vert^{-2}\langle\P^a_0\big(\supp(\ell_t)\subset\alpha_tG\big)\rangle&\leq\Pr(a_t\in B_1)+\sum_{n=1}^{M/\eps}\e^{-t\beta_t^{-1}\alpha_t^{-\frac d\eta}\chi^{\ssup{\rm d}}(\Z^d)^{(\eta+1)/\eta}(n\eps)^{-1/\eta}}\Pr(a_t\in A_n)\\
&\qquad+\e^{-t\beta_t^{-1}\alpha_t^{-\frac d\eta}\chi^{\ssup{\rm d}}(\Z^d)^{(\eta+1)/\eta}\eps^{-1/\eta}}\Pr(a_t\in B_2).
\end{aligned}
$$
For the environment terms, we use Lemma \ref{lem:ubound-environment} to calculate their asymptotic probabilities, noting that 
$$
t\beta_t^{-1}\alpha_t^{-\frac d\eta}=\beta_t^\eta\alpha_t^d=t^{\frac \eta{\eta+1}},
$$ 
by the choice of $\beta_t$ in \eqref{betadef}. Again, the condition $\alpha_t\ll t^{\frac \eta{d(\eta+1)}}$ ensures that $\beta_t\gg 1$. The remainder of the proof is now similar to the analogous part of the proof of the upper bound in Theorem~\ref{thm:ldp-growing-good}, which we do not spell out.

\subsection{Proof of Theorem \ref{thm:ldp-growing-bad}, lower bound}

\noindent For  any finite and connected set $B\subset\Z^d$ containing the origin and any sufficiently large $t$, we simply use that $B\subset\alpha_t G$ and apply Theorem~\ref{thm:ldp-finite}, to obtain 
$$
\limsup_{t\to\infty}t^{-\frac\eta{\eta+1}}\log\langle\P^a_0\big(\supp(\ell_t)\subset\alpha_tG\big)\rangle \geq -K_{\eta,D}\chi^{\ssup{\rm d}}(B),
$$ 
which is exactly \eqref{eqn:ldp-growing-bad-lbound}. To obtain the better lower bound in \eqref{eqn:ldp-growing-bad-lbound-exact} in the special case $\eta=d/2$, we apply \eqref{eqn:ldp-growing-bad-lbound} for any $[-n,n]\cap\Z^d$ for any $n\in\N$. It therefore suffices to show that $\limsup_{n\to\infty}\chi^{\ssup{\rm d}}([-n,n]\cap\Z^d)\leq \chi^{\ssup{\rm d}}(\Z^d)$ in the case $\eta=d/2$. This was shown in Lemma \ref{lem:disc-varprob-converge}.

\section{Proof of Theorem \ref{thm:ldp-quenched}}\label{sec:localtimes-quenched}

As in the proof of the LDP in Proposition~\ref{lem:rescaled-fix-ldp} via the G\"artner-Ellis theorem, the main work in proving Theorem \ref{thm:ldp-quenched} consists in proving asymptotic rescaling properties of the principal $\ell^2(B_t)$-eigenvalue, but this time of the random operator $\alpha_t^2\Delta^a+V_t$ in $B_t=\alpha_tG\cap\Z^d$ for large $t$, where the rescaled version $V_t$ of a bounded and continuous function $V$ was defined in \eqref{eqn:vtdef}. This is done using methods from  the field of spectral homogenisation, which provides an answer to this question that actually extends to the full spectrum, not only the largest eigenvalue. Recall that $G=(0,1)^d$ is the open unit cube and that the conductances are uniformly elliptic, i.e., stay in $(\lambda,1/\lambda)$ almost surely for some $\lambda\in(0,1)$. 

In Section~\ref{sec:localtimes-quenched-homogenisation}, we modify a powerful existing result on spectral homogenisation of $\Delta^a$ to fit our needs. In Section~\ref{sec:localtimes-quenched-ldp}, we use the modified result for a proof of Theorem \ref{thm:ldp-quenched}.

\subsection{Spectral homogenisation in the random conductance model}\label{sec:localtimes-quenched-homogenisation}

Let us introduce a number of notations and recall some important facts. Recall that $c_\eff$ is the diffusion constant of the limiting Brownian motion that appears in the invariance principle for RWRC. Denote by $A=c_\eff\Id$ the covariance matrix corresponding to the Brownian motion.  For some function $V\in \Ccal_{\rm b}(G)$, the set of bounded and continuous real-valued functions on $G$, let us consider the operator 
$$
-\frac 12\nabla^\ast A\nabla+V=-\frac{c_\eff}2\Delta+V
$$ 
defined on the Sobolev space $H_0^1(G)$. By the spectral theorem for elliptic operators (compare e.g.~Zimmer \cite{Z90}), the spectrum of this operator is given by a sequence $\lambda_1(V)<\lambda_2(V)\leq\lambda_3(V)\leq\ldots$ of eigenvalues (counted according to their multiplicity) with corresponding $L^2$-normed eigenfunctions $v_1,v_2,\ldots\in \Ccal_0^\infty(G)$. For $t\geq 0$, let $\lambda_1^{\ssup t}(V)<\lambda_2^{\ssup t}(V)\leq\lambda_3^{\ssup t}(V)\leq\ldots$ denote the eigenvalues of $-\alpha_t^2\Delta^a+V_t$ on $\ell^2(B_t)$ with zero boundary condition, where $V_t$ is defined in \eqref{eqn:vtdef} above. Then, let $v_1^{\ssup t},v_1^{\ssup t},\ldots$ be the corresponding normed eigenfunctions. The values $\lambda_j^{\ssup t}(V)$ and functions $v_j^{\ssup t}$ in the case that $j$ is larger than the dimension of $\ell^2(B_t)$, say $j_0$, are of no importance and we just define them to be equal to $\lambda_{j_0}^{\ssup t}(V)$ resp. $v_{j_0}^{\ssup t}$.

\begin{theorem}[Spectral homogenisation]\label{thm:spectral-homogenisation}
Fix $V\in \Ccal_{\rm b}(G)$. Then, for each $j\in\N$, as $t\to\infty$,
\begin{equation}\label{eqn:spec-homog-ef-conv}
\lambda_j^{\ssup t}(V)\longrightarrow\lambda_j(V)\qquad \mbox{and}\qquad\big\Vert v_j^{\ssup t}-\alpha_t^{-d/2}v_j\big(\smfrac{\cdot}{\alpha_t+1}\big)\big\Vert_2\to 0.
\end{equation}
\end{theorem}

This statement has been proven in the special case $V\equiv 0$ in \cite{BD03} with ideas going back to Kesavan (\cite{K79}). In order to generalise their result to cover the case of non-zero potential $V$, we state a version of an intermediate result from \cite{BD03} based on which we subsequently prove Theorem \ref{thm:spectral-homogenisation}. In the following, we tacitly extend any function $f\colon G\to\R$ trivially (i.e., with the value zero) to a function $f\colon \R^d\to\R$ and define $\hat f_n(z)=f(z/(n+1))$ for $z\in\Z^d$ and $n\in\N$.

\begin{lemma}\label{lem:spectral-homogenisation}
For $n\in\N$, let $u_n\in\ell_2(\Z^d)$ with $\supp (u_n)\subset nG$ and $\Vert u_n\Vert_2=1$. Assume that $n^2\Vert(\Delta^a u_n)\1_{nG}\Vert_2$ is bounded.

Then, almost surely, any subsequence $(n_k)_{k\in\N}$ of strictly increasing integers contains a further subsequence $(\hat n_k)_{k\in\N}$ such that there is a function $q\in H_0^1(G)$ such that for all $\varphi\in\Ccal(G)\cap L^2(G)$ and $f\in\{1\}\cup\{a(\cdot,e)\colon e\in\Ncal\}$ and for all $e\in\Ncal$, as $n\to\infty$ along $\hat n_k$,
\begin{eqnarray}
n^{-d/2}\sum_{z\in\Z^d}u_n(z)\hat\varphi_n(z)f(z)&\to&\langle f\rangle\int_Gq(y)\varphi(y)\,\d y,
\label{eqn:hom-lem-1}\\
n^{(2-d)/2}\sum_{z\in\Z^d}a(z,e)(u_n(z+e)-u_n(z))\hat\varphi_n(z)&\to& c_\eff\int_G\partial_eq(y)\varphi(y)\,\d y.\label{eqn:hom-lem-2}
\end{eqnarray}
If the function $q$ is continuous, we have in addition
\begin{equation}\label{eqn:hom-lem-3}
\Vert u_n-n^{-d/2}\hat q_n\Vert_2\to 0\quad\text{as }t\to\infty.
\end{equation}
\end{lemma}

This result already encapsulates the input from homogenisation theory and ergodic theory. We turn to the proof of Theorem \ref{thm:spectral-homogenisation} following the same route as the the proof of the analogous result for $V\equiv0$ in \cite{BD03}.

\begin{proof}[Proof of Theorem \ref{thm:spectral-homogenisation}]
Write $\lambda_j^{\ssup t}$ and  $\lambda_j$ instead of $\lambda_j^{\ssup t}(V)$ and $\lambda_j(V)$. As we consider subsets of the lattice, we may, without loss of generality, assume that $\alpha_t$ takes integer values only. With  $\mu_1^{\ssup t},\mu_1^{\ssup t},\ldots$ the Dirichlet eigenvalues of the homogeneous discrete operator $-\frac 12\Delta$ on $\alpha_tG\cap\Z^d$,  the eigenfunctions $v_j^{\ssup t}$, $j\in\N$ clearly satisfy
\begin{equation}
\alpha_t^2\Vert(\Delta^a v_j^{\ssup t})\1_{\alpha_t G}\Vert_2\leq\lambda_j^{\ssup t}\leq  \frac{\alpha_t^2}\lambda\mu_j^{\ssup t},
\end{equation}
where $\lambda\in(0,1)$ is the ellipticity parameter for the conductances. As the eigenvalues $\mu_j^{\ssup t}$ are known to be of order $\alpha_t^{-2}$, the $v_j^{\ssup t}$ satisfy the prerequisites of Lemma \ref{lem:spectral-homogenisation} and we may conclude that, for $j\in\N$, there are $\nu_j\in\R$ and $q_j\in H_0^1(G)$ such that for all $\varphi\in\Ccal(G)\cap L^2(G)$, as $ t\to\infty$,
\begin{eqnarray}
\lambda_j^{\ssup t}&\rightarrow& \nu_j,\label{eqn:spec-homog-ev-converge}\\
\alpha_t^{-d/2}\sum_{z\in\Z^d}v_j^{\ssup t}(z)\hat\varphi_{\alpha_t}(z)&\rightarrow&\int_Gq_j(y)\varphi(y)\,\d y,\label{eqn:spec-homog-weak-converge}\\
\alpha_t^{(2-d)/2}\sum_{z\in\Z^d}a(z,e)(v_j^{\ssup t}(z+e)-v_j^{\ssup t}(z))\hat\varphi_{\alpha_t}(z)&\rightarrow& c_\eff\int_G\partial_eq_j(y)\varphi(y)\,\d y.\label{eqn:spec-homog-weak-grad-converge}
\end{eqnarray}
Let us show that the $\nu_j$ are eigenvalues of $-\frac{c_\eff}{2}\Delta+V$ with corresponding eigenfunction $q_j$. Indeed, for all $\varphi\in\Ccal_0^\infty(G)$, by \eqref{eqn:spec-homog-weak-converge},
\begin{equation}\label{eqn:spec-homog-ev-1}
\alpha_t^{-d/2}\sum_{z\in\Z^d}\big((-\alpha_t^2\Delta^a+V_t)v_j^{\ssup t}(z)\hat\varphi_{\alpha_t}(z)\big)=\lambda_j^{\ssup t} \alpha_t^{-d/2}\sum_{z\in\Z^d}v_j^{\ssup t}(z)\hat\varphi_{\alpha_t}(z)\underset{t\to\infty}\longrightarrow\nu_j\int_Gq_j(y)\varphi(y)\,\d y.
\end{equation}
On the other hand, by \eqref{eqn:spec-homog-weak-converge}, \eqref{eqn:spec-homog-weak-grad-converge} and integration by parts (using symmetry of the conductances),
\begin{align}\label{eqn:spec-homog-ev-2}
&\alpha_t^{-d/2}\sum_{z\in\Z^d}\big((-\alpha_t^2\Delta^a+V_t)v_j^{\ssup t}(z)\hat\varphi_{\alpha_t}(z)\big)\notag\\
=&-\smfrac 12\alpha_t^{(2-d)/2}\sum_{z\in\Z^d}\sum_{e\in\Ncal}a(z,e)\Big[(v_j^{\ssup t}(z+e)-v_j^{\ssup t}(z))\alpha_t\big(\hat\varphi_{\alpha_t}(z+e)-\hat\varphi_{\alpha_t}(z)\big)\Big]\notag\\
&+\alpha_t^{-d/2}\sum_{z\in\Z^d}\big(V_t(z)v_j^{\ssup t}(z)\hat\varphi_{\alpha_t}(z)\big)\notag\\
\underset{t\to\infty}\longrightarrow &-\frac{c_\eff}{2}\sum_{e\in\Ncal}\int_G\partial_e q_j(y)\partial_e\varphi(y)\,\d y+\int_Gq_j(y)V(y)\varphi(y)\,\d y.
\end{align}
In the last step, we also used that $\alpha_t\big(\hat\varphi_{\alpha_t}(z+e)-\hat\varphi_{\alpha_t}(z)\big)-\widehat{\partial_e\varphi}_{\alpha_t}(z)$ as well as $V_t(z)-\hat V_{\alpha_t}(z)$ vanish at least in a weak $L^2$-sense. The limits in \eqref{eqn:spec-homog-ev-1} and \eqref{eqn:spec-homog-ev-2} show that the left-hand sides of these two are equal, which means the $\nu_j$ are eigenvalues of $-\frac{c_\eff}{2}\Delta+V$ with eigenfunction $q_j$. It now remains to show that the $\nu_j$ are in fact all eigenvalues of that operator and therefore constitute the entire $H_0^1$-spectrum. This is done in complete analogy with \cite{BD03}, Corollary 2, hence we omit it here for conciseness. As the eigenvalues $\lambda_j^{\ssup t}$ are ordered, so are the $\nu_j$. This means we have, for all $j\in\N$, $\lambda_j^{\ssup t}\to\nu_j=\lambda_j$ as $t\to\infty$ and $q_j=v_j$. Finally, as the $v_j$ are continuous, \eqref{eqn:spec-homog-ef-conv} follows from \eqref{eqn:hom-lem-3} in Lemma~\ref{lem:spectral-homogenisation}.
\end{proof}

\subsection{Proof of Theorem \ref{thm:ldp-quenched}}\label{sec:localtimes-quenched-ldp}

The proof is conducted in analogy with the proof of Proposition~\ref{lem:rescaled-fix-ldp} in Section \ref{sec:ldp-growing-prelim}. Like in that proof, it will be sufficient to show that
\begin{equation}\label{eqn:homog-cumulant-generating}
\lim_{t\to\infty}\frac{\alpha_t^2}{t}\log\E_z^{a}\Big[\exp\Big\{-\frac{t}{\alpha_t^2}\int_G V(y)L_t(y)\,\d y\Big\}\,\Big\vert\, X_{[0,t]}\subset\alpha_tG\Big]=-\lambda_1(V)+\lambda_1(0),
\end{equation}
for all $V\in \Ccal_{\rm b}(G)$.  For such a $V$, define the operator $\Pcal_t^{a,V}$ on $\ell^2(\alpha_tG\cap\Z^d)$ by
\begin{equation*}
\Pcal_t^{a,V}f(z)=\E_z^{a}\Big[\exp\Big\{-\frac{t}{\alpha_t^2}\int_G V(y)L_t(y)\,\d y\Big\}\1\{X_{[0,t]}\subset\alpha_tG\}f(X_t)\Big].
\end{equation*}
Then, \eqref{eqn:homog-cumulant-generating} is implied by showing
\begin{equation*}
\lim_{t\to\infty}\frac{\alpha_t^2}{t}\log\Pcal_t^{a,V}\1(0)=-\lambda_1(V)
\end{equation*}
instead. Recalling the definitions \eqref{eqn:vtdef} of $V_t$ and \eqref{eqn:rescaled_local_times} of $L_t$, we see that
\begin{align*}
\Pcal_t^{a,V}f(z)=\E_z^{a}\Big[\exp\Big\{-\frac{1}{\alpha_t^2}\int_0^t V_t(X_s)\,\d s\Big\}\1\{X_{[0,t]}\subset\alpha_tG\}f(X_t)\Big].
\end{align*}
Consequently, $\Pcal_t^{a,V}$ admits the semigroup representation
$$
\Pcal_t^{a, V}=\exp\{t(\Delta^{a}-\alpha_t^{-2}V_t)\}=\exp\Big\{-t\alpha_t^{-2}\big[-\alpha_t^2\Delta^{a}+V_t\big]\Big\},
$$
where the operator in the exponent is considered in $\ell^2(\alpha_tG\cap\Z^d)$ with zero boundary condition. Note that $\Pcal_t^{a, V}$ has the same principal eigenfunction as the operator $-\alpha_t^2\Delta^{\varphi_t}+V_t$ has, and the corresponding principal eigenvalue is given by $\exp\big\{-\frac{t}{\alpha_t^2}\lambda_1^{\ssup t}(V)\big\}$. An eigenvalue expansion yields, for each $t\geq 0$,
\begin{equation*}
\exp\Big\{-\frac{t}{\alpha_t^2}\lambda_1^{\ssup t}(V)\Big\}\big(v_t(0)\big)^2\leq \Pcal_t^{a,V}\1(0)\leq\vert\alpha_tG\vert^2\exp\Big\{-\frac{t}{\alpha_t^2}\lambda_1^{\ssup t}(V)\Big\}.
\end{equation*}
By Theorem \ref{thm:spectral-homogenisation}, $\lambda_1^{\ssup t}(V)\to\lambda_1(V)$ as $t\to\infty$, so it remains to show that $v_t(0)$ decays only polynomially as $t\to\infty$. Since $v_t$ is an eigenfunction of $-\alpha_t^2\Delta^{a}+V_t$ corresponding to the eigenvalue $\lambda_1^{\ssup t}(V)$, we have
\begin{align*}
v_t(0)&=\e^{-\lambda_1^{\ssup t}(V)}\big(\exp\{\alpha_t^2\Delta^a-V_t\}v_t\big)(0)\\
&=\e^{-\lambda_1^{\ssup t}(V)}\E^{\alpha_t^2a}_0\Big[\exp\Big\{-\int_0^1V_t(X_s)\,\d s\Big\}v_t(X_1)\Big]\\
&\geq\e^{-\lambda_1^{\ssup t}(V)-V_\ast}\E^{\alpha_t^2a}_0\big[v_t(X_1)\big]
\end{align*}
where $V_\ast$ is some upper bound for $V$. Abbreviating $v_t^\ast=\max_{x\in\alpha_tG\cap\Z^d} v_t(x)$, we estimate
\begin{equation*}
v_t(0)\geq v_t^\ast\e^{-\lambda_t(V)-V_\ast}\min_{x\in B_t}\P^{\alpha_t^2a}_0\big(X_1=x\big).
\end{equation*}
As $v_t$ is normed, the decay of its maximal value is only polynomial as $t\to\infty$, so we only need to consider the exponential decay of the probability term above. Here we employ a heat kernel estimate from \cite[Theorem 1.2]{BD10}. It says that there are positive constants $c_1,c_2$ such that, for $t$ sufficiently large (depending on the realisation of the conductances),
\begin{equation*}
\P^{\alpha_t^2a}_0\big(X_1=x\big)=\P^{a}_0\big(X_{\alpha_t^2}=x\big)\geq c_1\alpha_t^{-d}\e^{-c_2\vert x\vert^2/\alpha_t^2}
\end{equation*}
for all $x\in\Z^d$ with $\vert x\vert\leq\alpha_t^3$. As $\vert x\vert^2/\alpha_t^2$ is bounded, we have shown that $v_t(0)$ decays only polynomially as $t\to\infty$, and the proof of Theorem \ref{thm:ldp-quenched} is finished.

\end{document}